\RequirePackage{fix-cm}
\documentclass[smallextended]{svjour3}       % onecolumn (second format)
\smartqed  % flush right qed marks, e.g. at end of proof
\usepackage{graphicx}
\usepackage{amsmath}
\usepackage{amssymb}
\usepackage{mathtools}
\usepackage{mlmath}
\usepackage{booktabs}
\usepackage{hyperref}

\usepackage{amssymb,pifont}
\usepackage{fontawesome5}

\usepackage{algorithm}
\usepackage{algorithmic}
\newtheorem{assumption}[theorem]{Assumption}

\DeclareMathOperator*{\argmin}{arg\,min}
\begin{document}

\title{A SAV Neural Network Method for Allen--Cahn and Cahn--Hilliard Equations
}
% \subtitle{Do you have a subtitle?\\ If so, write it here}


\author{\mbox{Tao Luo \textsuperscript{1,2 *}\and Guihong Wang \textsuperscript{1}\and Wei Zhao \textsuperscript{1 $\dagger$}}}
\authorrunning{Tao Luo et al.}


\institute{Tao Luo \at
              \email{luotao41@sjtu.edu.cn}           %  \\
           \and
           Guihong Wang \at
              \email{theodore123@sjtu.edu.cn}
            \and 
            Wei Zhao \at
            \email{zhaowei01@sjtu.edu.cn}
            \and
\textsuperscript{\rm 1} School of Mathematical Sciences, Shanghai Jiao Tong University, Shanghai, China \\
\textsuperscript{\rm 2} Institute of Natural Sciences, MOE-LSC, CMA-Shanghai, Shanghai Jiao Tong University, Shanghai, China \\
\textsuperscript{\rm *} Corresponding author \\
\textsuperscript{\rm $\dagger$} In alphabetic order
}


\maketitle

\begin{abstract}
Phase-field equations are fundamental examples 
of gradient-flow systems and require numerical solvers that preserve intrinsic energy dissipation.
In this work, we propose a neural network-based method that combines an SAV time
discretization with a mesh-free neural representation.
We establish the unconditional energy stability of the underlying SAV semi-discretization,
quantify the energy defect caused by solver residuals, and derive a conditional
high-probability error estimate for the Allen--Cahn equation. Numerical experiments show that our approach 
consistently exhibits energy-stable behavior, while outperforming existing neural network-based solvers.
\end{abstract}

\section{Introduction}

Partial differential equations (PDEs) play a central role in a wide range of scientific and 
engineering disciplines, including fluid dynamics \cite{WendtComputational2009}, materials science \cite{LeeComputational2016}, phase transition phenomena \cite{CahnFree1958JCP,AllenMicroscopic1979AM} and
chemical reaction-diffusion systems \cite{SchatzQuantum2002}. These equations are 
typically derived by modeling fundamental physical and chemical principles, and serve as 
the primary mathematical framework for analyzing complex processes.
Among them, gradient flows form an important class of energy-driven dynamical systems, 
in which the temporal evolution of the state variable is governed by the variational 
derivative of an underlying energy functional. This structure naturally induces energy 
dissipation and provides a principled description of irreversible processes.

Phase-field equations are prominent gradient-flow models for interface dynamics.
Canonical examples are the Allen--Cahn and Cahn--Hilliard equations, both of
which are governed by a Lyapunov functional commonly referred to as the free energy.
In particular, the Allen--Cahn equations correspond to an $L^2$-gradient flow, whereas the Cahn--Hilliard equations can 
be formulated as an $H^{-1}$-gradient flow. These nonlinear equations exhibit rich multi-scale behavior and often involve 
stiff dynamics due to the presence of small interfacial parameters, posing significant challenges for both 
classical numerical solvers and learning-based approaches.
% \iffalse
% \begin{figure}[ht]
%   \vskip 0.2in
%   \begin{center}
%     \centerline{\includegraphics[width=0.8\columnwidth]{learning_process_new.png}}
%     \caption{
%       Learning procedures
%     }
%   \label{fig:intro}
%   \end{center}
% \end{figure}
% \fi
\begin{figure}[ht]
  % \vskip 0.1in
  \begin{center}
    \centerline{\includegraphics[width=0.5\columnwidth]{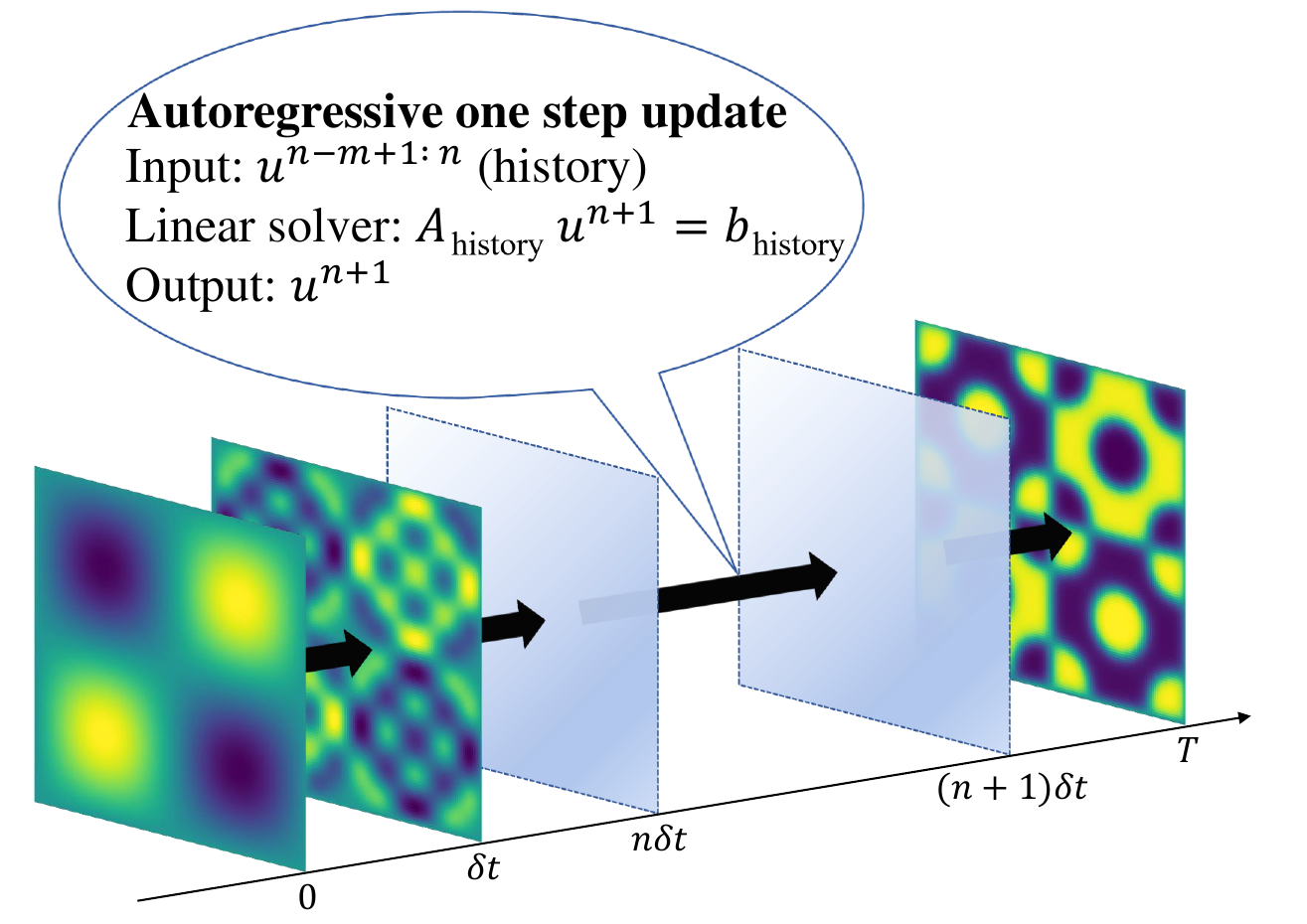}}
    \caption{
      Illustration of the learning procedure of the proposed method based on physics-informed updates from history solutions.
    }
  \label{fig:intro}
  \end{center}
  \vskip -0.1in
\end{figure}

% A natural approach to solving phase-field equations is to treat space and time separately.
% Classical numerical methods first discretize the spatial domain—using spectral, finite element, 
% or finite difference techniques—and then apply conventional time integrators to advance the resulting system.
% While effective and well-understood, this paradigm also highlights the central role of time integration 
% in determining stability and long-time behavior, making it a natural target for learning-based alternatives.

Over the past decades, numerical solvers typically adopt a time-space splitting strategy, discretizing the evolution in time and then resolving the resulting subproblems in space. 
% via finite differences, finite elements, or spectral methods. 
To capture the energy-dissipation structure inherent in phase-field models, several representative classes of energy-stable schemes have been proposed. 
The convex splitting approach constructs energy-stable schemes by implicitly treating the convex part of the free energy and explicitly handling the concave part~\cite{ShenSecondorder2012SJNA}.
The invariant energy quadratization (IEQ) approach introduces auxiliary variables to transform nonlinear free-energy terms into quadratic forms, leading to linear schemes with guaranteed discrete energy decay~\cite{yang2017efficient}; 
the scalar auxiliary variable (SAV) approach can be viewed as a refinement of IEQ, replacing field-type auxiliary variables with a single scalar variable to simplify implementation while retaining unconditional energy stability~\cite{ShenScalar2018JCP,zhu2019efficient,huang2022new}; 
and variational time discretizations such as the Jordan--Kinderlehrer--Otto (JKO) minimizing-movement scheme advance the solution by solving a sequence of energy-plus-metric minimization problems, thereby enforcing energy dissipation at the discrete level \cite{carrillo2022primal,carrillo2024structure}.
% Classical solvers for phase-field models rely on carefully designed spatial discretizations 
% (e.g., finite difference/element or spectral methods) and, more importantly, structure-preserving time integrators 
% that respect intrinsic properties such as energy dissipation and, for conserved dynamics, mass conservation. 
% Representative strategies include semi-implicit schemes and energy-stable formulations (e.g., convex splitting or 
% auxiliary-variable approaches), which enable robust long-time simulations under stiff interfacial scales.

In recent years, an increasing number of neural network-based solvers 
for PDEs have been proposed. Among them, physics-informed neural networks 
(PINNs) \cite{RaissiPhysicsinformed2019JCP} and their numerous variants \cite{WangPirateNets2024JMLR,WangRespecting2024CMAME} have attracted particular attention. 
To improve the accuracy towards the Allen--Cahn and Cahn--Hilliard equations, temporal adaptive 
sampling method \cite{WightSolving2021CCP} and time marching method \cite{MatteyNovel2022CMAMEa} are proposed,
which are motivated by the strong temporal dependency 
inherent in gradient-flow equations. Consistently, prior work \cite{WangRespecting2024CMAME} has pointed out that PINNs training should follow physical 
causality, requiring accurate optimization at early times before progressing to later stages. These observations naturally motivate 
a sequential training strategy that aligns the learning process with the underlying temporal dynamics, which is
consistent with numerical time-discretization schemes. In contrast, a class of neural network-based 
discretization methods \cite{HeAdaptive2023EABE,ChenTENG2024P4ICML}, 
which are different from time-marching strategy in formulation, have been widely studied. The core idea 
is to discretize time and employ a sequence of networks to learn the solution at each time step and advance the solution in an iterative manner. 
To mitigate the issue of computational cost, prior works \cite{CalabroTime2023AMC,WangSemidiscrete2025} 
suggest updating only the outer-layer parameters during training, which significantly improves efficiency 
while maintaining competitive accuracy.

Moreover, recent studies have incorporated gradient-flow structure into neural solvers through energetic variational discretizations \cite{HuEVNN2024} and SAV-based evolution of neural parameters \cite{LinEvoKAN2025}. Sequential neural time integration has also been studied within a unified optimize-then-discretize and discretize-then-optimize framework \cite{ZhangSequential2024}. More closely related approaches include GTransNet-BDF, which combines fixed-feature collocation with stabilized BDF schemes and a mass-conserving projection \cite{DoanGTransNetCH2026}. The present work focuses on BDF2–SAV collocation least-squares solves, quantifying the solver-induced energy defect and establishing a conditional high-probability error estimate that accounts for residual and sampling errors.

\textbf{Contributions.}
In this work, we propose an efficient SAV-based neural method for the
Allen--Cahn and Cahn--Hilliard equations.
Our main contributions are summarized as follows:
\begin{enumerate}
  \item We propose a structure-preserving neural method based on an
  unconditionally energy-stable SAV semi-discretization, and quantify the
  energy defect caused by the practical least-squares solver.
    
  \item We demonstrate that the proposed method naturally extends to \emph{irregular computational domains},
  highlighting the flexibility of mesh-free neural representations.
  
  \item We provide \emph{rigorous stability and error analyses}, including a conditional high-probability error estimate for domains satisfying the required elliptic regularity assumptions.

\end{enumerate}

\section{Formulation}

We consider gradient flows of the form
\begin{equation}
\begin{aligned}
    \frac{\partial \phi}{\partial t}  &= - \mathcal{G} \mu,\\
    \mu &= \frac{\delta E}{\delta\phi},
\end{aligned}
\label{eq:gf}
\end{equation}
where \(\mathcal G\) is a self-adjoint, positive-semidefinite mobility operator. In particular, \(\mathcal G=I\) yields the Allen--Cahn equation, whereas \(\mathcal G=-\Delta\) yields the Cahn--Hilliard equation. The associated free energy is given by
\[
E(\phi)=\int_\Omega
\left(\frac12\phi\mathcal L\phi+F(\phi)\right)\,dx,
\]
where \(\mathcal L\) is a linear, self-adjoint, positive-semidefinite operator. For the classical double-well potential
\[
F(\phi)=\frac{(\phi^2-1)^2}{4\epsilon^2},
\]
small values of \(\epsilon\) give rise to stiff interfacial dynamics.

We focus on the gradient-flow equations defined by Eq.~\eqref{eq:gf} on a bounded domain \(\Omega\subset\mathbb R^d\), subject to either periodic boundary conditions or homogeneous Neumann boundary conditions for both \(\phi\) and \(\mu\), together with the initial condition
\[
\phi(\boldsymbol{x},0)=\phi^0(\boldsymbol{x}),
\qquad \boldsymbol{x}\in\Omega.
\]

% \begin{equation}
%   \begin{aligned}
%     \frac{\partial \phi}{\partial t}  &= - \mathcal{G} \mu,\\
%     \mu &= -\Delta \phi + F^{\prime}(\phi),
%   \end{aligned}
%   \label{eq:target}
% \end{equation}
\subsection{Temporal Discretization and SAV Approach}
In this work, we focus on the Euler method and the second-order backward differentiation formula (BDF2) 
for analysis and implementation. The proposed method readily extends to other time-integration schemes, 
including Adams--Moulton (AM), Crank--Nicolson (CN), and Runge--Kutta (RK) methods \cite{HairerSolving1993}.

Define
\[
E_1[v]:=\int_\Omega F(v)\,dx+C_0,
\]
where \(C_0\) is chosen sufficiently large so that
$E_1[v]\ge c_*>0$
for all states considered in this paper. The SAV approach introduces
the auxiliary variable $r=\sqrt{E_1[\phi]}$, which enables the construction
of unconditionally energy-stable time-stepping schemes. The reformulated
system reads
\begin{equation}
\label{eq:sav}
\begin{aligned}
  \phi_t &= - \mathcal{G}\mu, \\
  \mu &= \mathcal{L} \phi+\frac{r(t)}{\sqrt{\int_{\Omega}F(\phi) \D x+C_0}}F^{\prime}({\phi}), \\
  r_{t}&=\frac{1}{2\sqrt{\int_{\Omega}F(\phi) \D x+C_{0}}}\int_{\Omega}F^{\prime}({\phi})\phi_{t}\D x.
\end{aligned}
\end{equation}

Once the time domain is discretized as $t_n = n\delta t, n = 0, 1,\ldots,N_T$, where $\delta t$ is the time step size, we can apply 
the semi-implicit Euler method to Eq.~\eqref{eq:sav}, leading to the following semi-discrete form:
\begin{equation}
\begin{aligned}
  \frac{\phi^{n+1}-\phi^n}{\delta t} &= -\mathcal{G}\left(\mathcal{L}\phi^{n+1}+r^{n+1}h(\phi^{n})\right), \\
  \frac{r^{n+1}-r^n}{\delta t} &= \int_\Omega \frac{h(\phi^n)}{2}\frac{\phi^{n+1}-\phi^n}{\delta t}\D x ,
 \end{aligned}
 \label{eq:saveuler}
\end{equation}
where $h(\phi) = \frac{F^{\prime}(\phi)}{\sqrt{\int_{\Omega} F(\phi) \D x + C_0}}$, and the BDF2 version is as follows:

\begin{equation}
\begin{aligned}
  \frac{3\phi^{n+1}-4\phi^n+\phi^{n-1}}{2\delta t} &= -\mathcal{G}\mu^{n+1}, \\
  \mu^{n+1} &= \mathcal{L}\phi^{n+1}+ r^{n+1} h(\phi^{n+\frac{1}{2}}) , \\
  \frac{3r^{n+1}-4r^n+r^{n-1}}{2\delta t} &= 
  \int_\Omega\frac{h(\phi^{n+\frac{1}{2}})}{2}\frac{3\phi^{n+1}-4\phi^n+\phi^{n-1}}{2\delta t}\D x .
 \end{aligned}
 \label{eq:savbdf2}
\end{equation}
Typically, we choose $\phi^{n+1/2}=2\phi^{n}-\phi^{n-1}$.

\subsection{Random Neural Basis}
Neural networks provide a mesh-free representation with built-in automatic differentiation, 
which has led to their widespread adoption in neural PDE solvers \cite{ZhangTransferable2024JSC,HuangNeural2025CNSNS}. The random neural basis can be 
viewed as a class of neural networks with fixed random inner-layer parameters, where only the 
outer layer is learned. These models have been shown to achieve spectral accuracy for approximating 
smooth functions, while retaining the simplicity of linear training. Notable instantiations include 
the Random Feature Method \cite{ChenBridging2022JML}, Extreme Learning Machines \cite{WangExtreme2024CMAME}, and Random Neural Networks \cite{ShangRandomized2024JEM}.
A function represented by random neural bases can be written as
\begin{equation}
  u(\vx) = \sum_{j=1}^{M} c_j \varphi_{\vtheta_j}(\vx).
\end{equation}
We denote the neural network basis functions 
$\varphi_{j} \coloneqq \varphi_{\vtheta_j} $, $\vtheta \coloneqq \text{vec}( \vtheta_1,\cdots,\vtheta_M ) $, $\vc \coloneqq \text{vec}(c_1,\cdots,c_M)$ and let $M$ denote the number of basis functions, i.e.,\  the network width of the penultimate layer.

Given sampled data points $ \vX = \{\vx_i\}_{i=1}^{N} $, one can easily solve a linear equation $\mathcal{F}u=f$ by the following procedures.

\begin{algorithm}[htbp]
\caption{Solving the Linear System}\label{alg:lin}
\begin{algorithmic}
\STATE {\bfseries Input:} Data $\vX$, neural basis $\{\varphi_j\}_{j=1}^{M}$, operator $\mathcal{F}$ and $f$
\STATE {\bfseries 1.} Compute the Jacobian matrix 
  $$\vJ \coloneqq \frac{\partial (\mathcal{F}u)}{\partial \vc}(\vX) = \text{CONCAT} \left\{\mathcal{F}\varphi_{\vtheta_j}(\vX) \right\}_{j=1}^{M}$$
\STATE {\bfseries 2.} Compute the right-hand side vector $\vb = f(\vX)$ (assume that there is no constant term in $\mathcal{F}$)
\STATE {\bfseries 3.} Solve the following least-squares problem
$$\hat{\vc} = \argmin_{\vc} \| \vJ\vc - \vb \|_2^2 $$
\STATE {\bfseries Output:} $\hat{u} = \sum_{j=1}^{M} \hat{c}_j\varphi_{j}$
\end{algorithmic}
\end{algorithm}
Note that $\vJ$ in Algorithm~\ref{alg:lin} is independent of $\vc$ and is determined solely by $\vtheta$ and $\vX$, 
thus can be precomputed by automatic differentiation and reused in temporal iterations.

\subsection{SAV-Based Neural Network Solver}
As illustrated in Fig.~\ref{fig:intro}, the learning task of the proposed method is to solve $\phi^{n+1}$ and $r^{n+1}$
from the history $\{\phi^{n},r^{n},\ldots,\phi^{n-m+1},r^{n-m+1}\}$,
where $m=1$ for the Euler method and $m=2$ for the BDF2 scheme. We
abstract systems~\eqref{eq:saveuler} and~\eqref{eq:savbdf2} as
\begin{equation}
  \label{eq:sys2sol}
  \begin{aligned}
    & \mathcal{L}_1(\phi^{n+1}, r^{n+1}) = f^n, \\
    & \mathcal{L}_2(\phi^{n+1}, r^{n+1}) = g^n.
  \end{aligned}
\end{equation}
For the Euler method, 
\begin{align*}
 \mathcal{L}_1(\phi^{n+1}, r^{n+1}) &= \left(\frac{1}{\delta t} + \mathcal{GL}\right)\phi^{n+1} + \mathcal{G}h^{n}r^{n+1}, &f^n&=\frac{\phi^n}{\delta t},\\
  \mathcal{L}_2(\phi^{n+1}, r^{n+1}) &= r^{n+1} - \frac{1}{2}\left\langle  h^{n}, \phi^{n+1} \right\rangle, &g^n&=r^n-\frac12\langle h^n,\phi^n\rangle,
\end{align*}
and for the BDF2 scheme,
\begin{align*}
 \mathcal{L}_1(\phi^{n+1}, r^{n+1}) &= \left(\frac{3}{2\delta t} + \mathcal{GL} \right)\phi^{n+1} + \mathcal{G}h^{n+\frac{1}{2}}r^{n+1}, \\
 f^n&=\frac{4\phi^n-\phi^{n-1}}{2\delta t},\\
  \mathcal{L}_2(\phi^{n+1}, r^{n+1}) &= 3 r^{n+1} - \frac{3}{2}\left\langle  h^{n+\frac{1}{2}}, \phi^{n+1} \right\rangle, \\
  g^n&=4r^n-r^{n-1}
 -\frac12\langle h^{n+1/2},4\phi^n-\phi^{n-1}\rangle, 
\end{align*}
where $h^{n}=h(\phi^n)$, $h^{n+\frac{1}{2}}=h(\phi^{n+\frac{1}{2}})$, and
$\langle\alpha,\beta\rangle=\int_{\Omega}\alpha\beta\D x$. 
%Through out theerror analysis, the integrals defining $E_1$, $h$, and the scalar SAV equation are assumed to be evaluated exactly; thus, Theorem~\ref{thm:error} quantifies collocation error but not numerical-quadrature error.
%As this system is linear for solution variable $\phi^{n+1}, r^{n+1}$, since the nonlinear terms $f^n$ and $g^n$ 
%involve only available states from previous time steps. We can solve Eq.~\eqref{eq:sys2sol} using 
%Algorithm~\ref{alg:lin} and we now present the detailed algorithm below.
The system~\eqref{eq:sys2sol} is linear in the unknowns $\phi^{n+1}$ and $r^{n+1}$, as the right-hand side terms $f^n$ and $g^n$ are evaluated using only available data from preceding time steps.
Thus, Algorithm~\ref{alg:lin} can be applied at each time level, as summarized in Algorithm~\ref{alg:evo}.
\begin{algorithm}[h]
\caption{Solving the Time-Evolution System}\label{alg:evo}
\begin{algorithmic}
\STATE {\bfseries Input:} Data $\vX=\{x_i\}_{i=1}^N$, initial state $\phi^0$, operator $\mathcal{G}$,
neural basis $\{\varphi_j\}_{j=1}^{M}$, time step size $\delta t$, number of time steps $N_T$
\STATE {\bfseries Initialization:} Represent or fit $\hat{\phi}^0$, set
$\hat r^0=\sqrt{E_1[\hat\phi^0]}$, and compute
$(\hat\phi^1,\hat r^1)$ from the Euler--SAV system using
Algorithm~\ref{alg:lin}
\FOR{$n=1$ {\bfseries to} $N_T-1$}
\IF{$n=1$ or $\{\varphi_j\}_{j=1}^{M}$ change or $\vX$ changes}
\STATE Compute Jacobian matrix $\vJ^{n+1}$ according to Eq.~\eqref{eq:sys2sol}
\ELSE
\STATE Update the history-dependent entries of $\vJ^{n+1}$
\ENDIF
\STATE Compute $\vb^{n+1}=\operatorname{CONCAT}\{f^n(\vX),g^n\}$
\STATE Solve the following problem 
  $$(\hat{\vc}^{n+1}_{\phi}, \hat{c}^{n+1}_{r}) = \argmin_{\vc = \mathrm{vec}({\vc}^{n+1}_{\phi}, {c}^{n+1}_{r})} \| \vJ^{n+1}\vc - \vb^{n+1} \|_2^2 $$
  and obtain $\hat{\phi}^{n+1} = \sum_{j=1}^{M} \hat{c}_{\phi,j}^{n+1}\varphi_{j}$, $\hat{r}^{n+1} = \hat{c}_r^{n+1}$
\ENDFOR
\STATE {\bfseries Output:} $\{(\hat\phi^n,\hat r^n)\}_{n=0}^{N_T}$
\end{algorithmic}
\end{algorithm}

Boundary conditions, denoted by $\mathcal B\phi^n=0$, may be incorporated as
additional residual rows and are omitted from the algorithm for brevity.
Algorithm~\ref{alg:evo} presents the BDF2 scheme; the Euler scheme follows a
similar procedure. The computational cost per time step mainly stems from
constructing the Jacobian matrix $\vJ^{n+1}$ and solving the linear
least-squares problem, with the former being the dominant expense.
As outlined in Algorithm~\ref{alg:evo}, the Jacobian need not be fully
recomputed at every step; only a few entries require updating, as illustrated
in Fig.~\ref{fig:jac}.
\begin{figure}[ht]
  % \vskip 0.2in
  \begin{center}
    \centerline{\includegraphics[width=0.5\columnwidth]{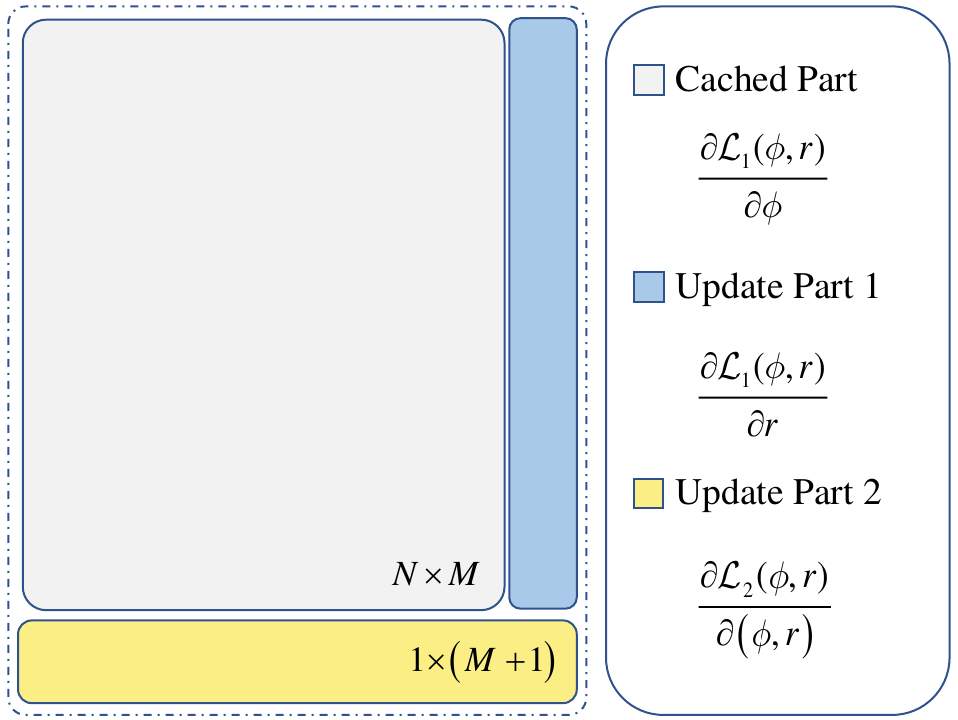}}
    \caption{
      The update of Jacobian matrix in the iterations of Algorithm~\ref{alg:evo}. Only the blue and yellow blocks have to
      be updated during iterations.
    }
  \label{fig:jac}
  \end{center}
  \vskip -0.2in
\end{figure}

\section{Energy Stability and Error Analysis}\label{sec:theory}

In this section, we establish the main theoretical properties of the proposed method. 
We first derive the energy laws satisfied by the exact BDF2--SAV
semi-discrete solution and by the neural numerical solution, respectively.
We then derive error estimates for the least-squares solver used to solve the linear subproblem at each time step in Algorithm~\ref{alg:evo}.

\subsection{Energy Laws for the Semi-Discrete and Numerical Solutions}\label{sec:energy-stability}

A principal advantage of the SAV formulation is that the resulting time-discrete
scheme inherits a modified energy law. We first recall the standard result for an
exact solution sequence of the BDF2--SAV semi-discretization.

\begin{theorem}[Energy stability of the BDF2--SAV semi-discretization, Ref.~\cite{shen2019new}]\label{thm:energy-stability}
Let
$\{(\phi^n,\mu^n,r^n)\}_{n\geq0}$ be a spatially continuous solution sequence
that satisfies the time-discrete scheme~\eqref{eq:savbdf2} exactly. Then, for
every $\delta t>0$ and $n\geq1$,
\begin{equation*}
\begin{aligned}
&\bar{E}^{n+1}-\bar{E}^{n}\\
&= -\delta t\,\left\langle\mathcal G\mu^{n+1},\mu^{n+1}\right\rangle
-\frac{1}{2}\left(r^{n+1}-2r^n+r^{n-1}\right)^2 \\
&\quad-\frac{1}{4}\left\langle
\mathcal L(\phi^{n+1}-2\phi^n+\phi^{n-1}),
\phi^{n+1}-2\phi^n+\phi^{n-1}\right\rangle,
\end{aligned}
\end{equation*}
where
\begin{equation*}
\begin{aligned}
\bar E^{n+1}
&=\frac12\left[\left(r^{n+1}\right)^2
+\left(2r^{n+1}-r^n\right)^2\right]\\
&\quad+\frac14\left[\left\langle\mathcal L\phi^{n+1},\phi^{n+1}\right\rangle
+\left\langle\mathcal L(2\phi^{n+1}-\phi^n),
2\phi^{n+1}-\phi^n\right\rangle\right].
\end{aligned}
\end{equation*}
In particular, $\bar E^{n+1}\leq \bar E^n$ for every $\delta t>0$.
\end{theorem}

Theorem~\ref{thm:energy-stability} concerns the exact solution of the
time-discrete BDF2--SAV scheme.
For the practical least-squares neural solver in Algorithm~\ref{alg:evo}, however, nonzero residuals are generally present, so the computed solution does not obey the same energy law.
The following result provides a quantitative characterization of this energy defect.
\begin{theorem}[Energy balance for the BDF2--SAV neural solution]
\label{thm:residual-energy-law}
Let $\{(\widehat\phi^n,\widehat r^n)\}_{n\geq0}$ denote the neural numerical
solution sequence associated with the time-discrete BDF2--SAV system. For
$n\geq1$, set
\[
\delta_2v^{n+1}:=3v^{n+1}-4v^n+v^{n-1},
\qquad
\widehat h^{n+\frac12}:=h(2\widehat\phi^n-\widehat\phi^{n-1}),
\]
and reconstruct the chemical potential by
\[
\widehat\mu^{n+1}:=\mathcal L\widehat\phi^{n+1}
+\widehat r^{n+1}\widehat h^{n+\frac12}.
\]
Assume that $\widehat\phi^{n-1}$, $\widehat\phi^n$,
$\widehat\phi^{n+1}$, and $\widehat\mu^{n+1}$ satisfy the boundary condition. Define the continuous
residuals $\rho_\phi^{n+1}$ and $\rho_r^{n+1}$ by
\begin{align}
\delta_2\widehat\phi^{n+1}
&=-2\delta t\,\mathcal G\widehat\mu^{n+1}+\rho_\phi^{n+1},
\label{eq:neural-phase-residual}\\
\delta_2\widehat r^{n+1}
&=\frac12\left\langle\widehat h^{n+\frac12},
\delta_2\widehat\phi^{n+1}\right\rangle+\rho_r^{n+1}.
\label{eq:neural-sav-residual}
\end{align}
Define $\|v\|_{\mathcal L}^2:=\langle\mathcal Lv,v\rangle$ and
\begin{align*}
\widehat{\bar E}^{n+1}
&:=\frac12\left[\left(\widehat r^{n+1}\right)^2
+\left(2\widehat r^{n+1}-\widehat r^n\right)^2\right] + \frac14\left[\|\widehat\phi^{n+1}\|_{\mathcal L}^2
+\|2\widehat\phi^{n+1}-\widehat\phi^n\|_{\mathcal L}^2\right].
\end{align*}
Then the modified energy satisfies
\begin{align}
\widehat{\bar E}^{n+1}-\widehat{\bar E}^{n}
&=-\delta t\,\left\langle\mathcal G\widehat\mu^{n+1},
\widehat\mu^{n+1}\right\rangle
-\frac14\|\widehat\phi^{n+1}-2\widehat\phi^n
+\widehat\phi^{n-1}\|_{\mathcal L}^2\nonumber\\
&\quad-\frac12\left(\widehat r^{n+1}-2\widehat r^n
+\widehat r^{n-1}\right)^2
+\frac12\left\langle\rho_\phi^{n+1},\widehat\mu^{n+1}\right\rangle
+\widehat r^{n+1}\rho_r^{n+1}.
\label{eq:residual-energy-law}
\end{align}
Consequently, the modified energy is non-increasing whenever the combined
residual work in the last two terms of~\eqref{eq:residual-energy-law} does not
exceed the total nonnegative dissipation represented by the preceding three
terms.
\end{theorem}

\begin{proof}
Take the inner product of~\eqref{eq:neural-phase-residual} with
$\widehat\mu^{n+1}$ and use its reconstruction above. Multiplying
\eqref{eq:neural-sav-residual} by $2\widehat r^{n+1}$ produces the same
nonlinear term, which can therefore be eliminated. This gives
\begin{align*}
&\left\langle\mathcal L\widehat\phi^{n+1},
\delta_2\widehat\phi^{n+1}\right\rangle
+2\widehat r^{n+1}\delta_2\widehat r^{n+1}\\
&\quad=-2\delta t\left\langle\mathcal G\widehat\mu^{n+1},
\widehat\mu^{n+1}\right\rangle
+\left\langle\rho_\phi^{n+1},\widehat\mu^{n+1}\right\rangle
+2\widehat r^{n+1}\rho_r^{n+1}.
\end{align*}
For any symmetric positive-semidefinite bilinear form $B(\cdot,\cdot)$, let
$\|v\|_B^2:=B(v,v)$. The BDF2 identity
\[
\begin{aligned}
2B(a,3a-4b+c)
&=\|a\|_B^2+\|2a-b\|_B^2-\|b\|_B^2\\
&\quad-\|2b-c\|_B^2+\|a-2b+c\|_B^2
\end{aligned}
\]
applied with $B(a,b)=ab$ for $r$ and
$B(u,v)=\langle\mathcal Lu,v\rangle$ for $\phi$ proves
\eqref{eq:residual-energy-law}.
\end{proof}

Theorem~\ref{thm:residual-energy-law} reduces to Theorem~\ref{thm:energy-stability} when all residuals vanish. 
It assumes that the trial functions satisfy the boundary conditions exactly; otherwise, an additional boundary-work term appears. 
A collocation analogue, \(\widehat{\bar E}_N^{n+1}\), can be defined by replacing continuous inner products with empirical ones. 
If the discrete operators remain symmetric and positive semidefinite, the same argument yields a discrete energy balance and bounds the residual-work terms by discrete residual norms. 
Under suitable stability, norm-equivalence, and boundedness assumptions, the least-squares residual may control the discrete energy defect. 
We do not pursue this result because the required operator properties cannot be guaranteed for general collocation points.

\subsection{Error estimates}\label{sec:error-estimates}

In this section, we derive error estimates for the proposed solver by comparing the semi-discrete numerical solution with the exact solution of Eq.~\eqref{eq:sav}. 
To highlight the main ideas, we restrict our analysis to the Allen--Cahn case, with \(\mathcal G=I\) and \(\mathcal L=-\Delta+\lambda I\), and consider the BDF2-SAV scheme \eqref{eq:savbdf2}. 
For the classical Ginzburg--Landau free energy
\begin{equation}\label{eq:GL_energy}
E(\phi)
=
\int_{\Omega}
\left[
\frac{1}{2}\lvert\nabla\phi\rvert^{2}
+
\frac{(\phi^{2}-1)^{2}}{4\epsilon^{2}}
\right]\,dx,
\end{equation}
the associated Allen--Cahn equation corresponds to $ \mathcal L=-\Delta$, i.e. \(\lambda=0\). For any fixed \(\beta>0\), however, Eq.~\eqref{eq:GL_energy} can be equivalently rewritten as
\begin{equation}\label{eq:GL_energy_shift}
E(\phi)
=
\int_{\Omega}
\left[
\frac{1}{2}\lvert\nabla\phi\rvert^{2}
+
\frac{\beta}{2\epsilon^{2}}\phi^{2}
+
\frac{(\phi^{2}-1-\beta)^{2}}{4\epsilon^{2}}
-
\frac{\beta^{2}+2\beta}{4\epsilon^{2}}
\right]\,dx.
\end{equation}
Accordingly, we set
\begin{equation}\label{eq:shifted-lambda-potential}
\lambda:=\frac{\beta}{\epsilon^2},
\qquad
F(v):=
\frac{(v^2-1-\beta)^2}{4\epsilon^2}
-
\frac{\beta^2+2\beta}{4\epsilon^2}.
\end{equation}
It follows that \(\mathcal L=-\Delta+\lambda I\), with \(\lambda>0\), and
\[
E(\phi)
=
\int_\Omega
\left(
\frac12\phi\mathcal L\phi+F(\phi)
\right)\,dx.
\]
The constant term in \(F\) does not affect the variational derivative and hence leaves the dynamics unchanged.
Therefore, without loss of generality, we may assume \(\lambda>0\) in the subsequent analysis.
\iffalse
This subsection establishes error estimates for the proposed scheme by comparing the semi-discrete numerical solution with the exact solution of Eq.~\eqref{eq:gf}. 
To clarify the presentation and highlight key ideas, we focus on the Allen--Cahn type equation (i.e., $\mathcal{G}=I$ and $\mathcal{L}=-\Delta +\lambda I$) and the BDF2-SAV scheme \eqref{eq:savbdf2}. 
The resulting bounds explicitly characterize the dependence of accuracy on the time step size, the regularity of the exact solution, the training hyperparameters, and the least-squares residual.
For the standard Ginzburg--Landau free energy
\[
E[\phi]=\int_{\Omega}\left(\frac12|\nabla\phi|^{2}+\frac{(\phi^{2}-1)^{2}}{4\epsilon^{2}}\right)\,dx,
\]
the linear operator $\mathcal{L}$ in the corresponding Allen--Cahn equation has \(\lambda=0\). Nevertheless, for any fixed \(\beta>0\), the same energy can be equivalently rewritten by completing the square as
\begin{equation*}
\begin{aligned}
E[\phi]&=\int_{\Omega}\left(\frac{|\nabla\phi|^{2}}{2}+\frac{\beta}{2\epsilon^{2}}\phi^{2}\right.\\
&\quad\left. +\frac{(\phi^{2}-1-\beta)^{2}}{4\epsilon^{2}}
-\frac{\beta^{2}+2\beta}{4\epsilon^{2}}\right)\,dx.
\end{aligned}
\end{equation*}
Since the last term is a constant independent of \(\phi\), this reformulation does not change the variational derivative (and hence the dynamics), while it introduces a strictly positive quadratic contribution. Therefore, without loss of generality, we may assume \(\lambda>0\) in the subsequent analysis.
\fi

Throughout this subsection, we assume that the nonlinear force $F'$ satisfies the following conditions.
\begin{assumption}[Polynomial growth assumptions on $F'$]\label{assump:F-smooth}
Let $F \in C^{3}(\mathbb{R})$. There exist constants $C>0$ and exponents $p,p'>0$ such that, for all $x\in\mathbb{R}$,
\begin{align*}
|F''(x)| &\le C\bigl(|x|^{p}+1\bigr), \quad 
\begin{cases}
p>0 \ \text{arbitrary}, & d=1,2,\\
0<p<4, & d=3,
\end{cases}\\ \nonumber
|F'''(x)| &\le C\bigl(|x|^{p'}+1\bigr), \quad 
\begin{cases}
p'>0 \ \text{arbitrary}, & d=1,2,\\
0<p'<3, & d=3,
\end{cases}
\end{align*}
where \(d\) denotes the spatial dimension of Eq.~\eqref{eq:gf}.
\end{assumption}
These assumptions are standard in phase-field analysis and weaker than a
global Lipschitz condition on $F'$. In particular, they include the
double-well potential and support the standard well-posedness and convergence
arguments~\cite{temam2012infinite}.

To quantify the semi-discretization error, we state the following error estimates without assuming that $F'$ is globally Lipschitz continuous.
\begin{lemma}[Temporal discretization error estimate, Ref.~\cite{shen2018convergence}]\label{lem:error-semi-discrete}
Let $T>0$ and $t_n=n\,\delta t$ for $n\ge 0$. Define
\[
e^{n}:=\phi^{n}-\phi(t_n),\qquad s^{n}:=r^{n}-r(t_n).
\]
Under Assumption~\ref{assump:F-smooth}, assume that $\phi^0\in H^3(\Omega)$ and
\[
\phi \in L^{\infty}(0,T; W^{1,\infty}), \quad
\phi_t \in L^{\infty}(0,T; H^1), \]
\[ 
\phi_{tt} \in L^2(0,T; H^3), \quad 
\phi_{ttt} \in L^2(0,T; L^2).
\]
Then there exists a constant $C>0$, independent of $\delta t$, such that for all $0\le n\le \left\lfloor \frac{T}{\delta t} \right\rfloor$,
$$\lambda\|e^n\|^2+\|\nabla e^n\|^2+(s^n)^2 \le C\,\delta t^{4}.$$
Here $C$ may depend on $T$, $\Omega$, $\phi^0$, and $\phi$.
\end{lemma}

We next quantify the additional error introduced by the numerical implementation.
Due to the semi-discrete-in-time nature of the method, the error at the $(n+1)$-th time step originates from two distinct sources:
\begin{enumerate}
    \item \textbf{Local step solver error}, caused by the finite neural trial space
    and the least-squares solve at finitely many collocation points;
    \item \textbf{Propagation error}, caused by using previously computed
    states as the history data for the current subproblem.
\end{enumerate}
Accordingly, the error evolution can be characterized by a recurrence relation in which the error at the current time step depends on the accumulated errors from previous steps and the local solver error. 

To advance the semi-discrete scheme \eqref{eq:savbdf2} from $t_n$ to
$t_{n+1}$, the numerical implementation solves a linear system. Given the
computed history $(\hat\phi^n,\hat r^n)$ and
$(\hat\phi^{n-1},\hat r^{n-1})$, the numerical solution
$(\hat\phi^{n+1},\hat r^{n+1})$ is obtained by applying the inexact
least-squares solver to the following system:
\begin{equation}\label{eq:system_actual}
\left\{
\begin{aligned}
\bigl(3-2\delta t\Delta+2\delta t\lambda\bigr)\hat\phi^{n+1}
+2\delta t\,\hat r^{n+1}q^{n+\frac12}
&=4\hat\phi^n-\hat\phi^{n-1}
\quad\text{in }\Omega,\\
\hat r^{n+1}-\frac12
\left\langle q^{n+\frac12},\hat\phi^{n+1}\right\rangle
&=\frac13\bigl(4\hat r^n-\hat r^{n-1}\bigr)\\
&\quad-\frac16\left\langle q^{n+\frac12},
4\hat\phi^n-\hat\phi^{n-1}\right\rangle,\\
\mathcal B\hat\phi^{n+1}&=0
\quad\text{on }\partial\Omega.
\end{aligned}
\right.
\end{equation}
where
\[
\begin{aligned}
\hat\phi^{n+\frac12}&:=2\hat\phi^n-\hat\phi^{n-1},\\
q^{n+\frac12}&:=h(\hat\phi^{n+\frac12})
=\frac{F'(\hat\phi^{n+\frac12})}{\sqrt{\int_{\Omega}F(2\hat{\phi}^{n}-\hat{\phi}^{n-1})\D x+C_0}}.
\end{aligned}
\]

For the analysis, let
$(\hat\phi^{n+1,*},\hat r^{n+1,*})$ denote the exact solution of
Eqs.~\eqref{eq:system_actual} formed from the same numerical history.
%The existence and uniqueness of such a solution are guaranteed by the solvability of the SAV linear subsystem.

\begin{lemma}[Stability with respect to inexact linear solves]
\label{lem:error-accumulation}
Fix $T>0$. Let $\{(\phi^n,r^n)\}$ be the exact BDF2--SAV sequence,
initialized by one Euler--SAV step, and let
$\{(\hat\phi^n,\hat r^n)\}$ be the corresponding numerical sequence with
the same initial data. For $j\geq2$, let
$(\hat\phi^{j,*},\hat r^{j,*})$ be the exact solution of the BDF2 linear
subproblem formed from the numerical history and set
\[
\epsilon_j:=\|\hat\phi^j-\hat\phi^{j,*}\|_{H^2}
+|\hat r^j-\hat r^{j,*}|.
\]
Define $\hat e^n:=\hat\phi^n-\phi^n$, $\hat s^n:=\hat r^n-r^n$, and
\[
\mathcal E_{\rm st}:=
\lambda\|\hat e^1\|^2+\|\nabla\hat e^1\|^2+|\hat s^1|^2.
\]
Suppose that Assumption~\ref{assump:F-smooth} holds, $\phi^0\in H^3(\Omega)$,
and, for some $K>0$,
\[
\max_{0\leq j\leq T/\delta t}
\bigl(\|\hat\phi^j\|_{H^2}+|\hat r^j|\bigr)\leq K.
\]
Then there exists $\delta t_0>0$ such that, for $0<\delta t\leq\delta t_0$
and $t_n\leq T$,
\begin{equation}\label{eq:inexact-solve-stability}
\lambda\|\hat e^n\|^2+\|\nabla\hat e^n\|^2+|\hat s^n|^2
\leq C\left(\mathcal E_{\rm st}
+\delta t^{-1}\sum_{j=2}^{n}\epsilon_j^2\right).
\end{equation}
The constant $C$ may depend on $T$, $K$,
and the exact solution, but is independent of $\delta t$ and the local solver
errors.
\end{lemma}

The start-up requirement $\mathcal E_{\rm st}=O(\delta t^4)$ is obtained when
the Euler--SAV linear system is solved with $O(\delta t^2)$ accuracy in
$H^2(\Omega)\times\mathbb R$. Lemma~\ref{lem:error-accumulation} then shows
how the subsequent BDF2 solver errors accumulate. Its proof is given in
Appendix~\ref{sec:proof error accumulation}.

We next combine the deterministic stability of one linear subproblem with an
explicit sampling estimate for the fixed random basis used by
Algorithm~\ref{alg:evo}.

\begin{lemma}[High-probability $H^2$ error bound for the least-squares subproblem]
\label{lem:solver-error}
Assume that:
\begin{enumerate}
\item Assumption~\ref{assump:F-smooth} holds, $d\leq3$,
$0<\delta t\leq1$, $0<\eta<1$, and $M,N\in\mathbb N$;
\item $\Omega$ is a periodic box or a bounded $C^{1,1}$ domain with homogeneous
Neumann boundary conditions, with boundary-conforming space
$H^2_{\mathcal B}(\Omega)$;
\item $\{\varphi_j\}_{j=1}^M\subset H^2_{\mathcal B}(\Omega)$ is fixed and,
with
\[
\begin{gathered}
\Phi(x):=(\varphi_1(x),\ldots,\varphi_M(x))^{\mathsf T},\\
K_0:=\sup_{x\in\Omega}\|\Phi(x)\|_2,
\qquad K_2:=\sup_{x\in\Omega}\|\Delta\Phi(x)\|_2,
\end{gathered}
\]
one has $0<K_0<\infty$ and $K_2<\infty$.
\end{enumerate}

Fix $B>0$, and let $a,b\in\mathbb R^M$ and
$c=(c_\phi,c_r)\in\mathbb R^{M+1}$ satisfy
\[
\|a\|_2,\ \|b\|_2,\ \|c\|_2\leq B.
\]
Define
\[
u_a:=a^{\mathsf T}\Phi,\qquad v_b:=b^{\mathsf T}\Phi,\qquad
\psi_{a,b}:=h(2u_a-v_b),\qquad f_{a,b}:=4u_a-v_b.
\]
For $\gamma\in\mathbb R$, consider
\begin{equation}\label{eq:system}
\begin{aligned}
\mathcal L_{\delta t}\phi+2\delta t\,r\psi_{a,b}&=f_{a,b}
\quad \text{in }\Omega,\\
r-\tfrac12\langle\psi_{a,b},\phi\rangle
&=\gamma-\tfrac16\langle\psi_{a,b},f_{a,b}\rangle,
\end{aligned}
\end{equation}
where $\mathcal L_{\delta t}:=(3+2\delta t\lambda)I-2\delta t\Delta$ and
all inner products are evaluated exactly.

Set $\phi_c:=c_\phi^{\mathsf T}\Phi$ and $r_c:=c_r$. Let
$X_1,\ldots,X_N$ be i.i.d. uniform samples from $\Omega$, and define
\begin{align}
\mathcal R_c(x;a,b)
&:=\mathcal L_{\delta t}\phi_c(x)
+2\delta t\,r_c\psi_{a,b}(x)-f_{a,b}(x),\nonumber\\
\mathcal J_N(c;a,b,\gamma)
&:=\left|r_c-\tfrac12\langle\psi_{a,b},\phi_c\rangle-\gamma
+\tfrac16\langle\psi_{a,b},f_{a,b}\rangle\right|^2\nonumber\\
&\quad+\frac1N\sum_{i=1}^N|\mathcal R_c(X_i;a,b)|^2,
\label{eq:empirical-loss}
\end{align}
Then:
\begin{enumerate}
\item for every admissible $a,b,\gamma$, system~\eqref{eq:system} has a unique
solution $(\phi^*,r^*)\in H^2_{\mathcal B}(\Omega)\times\mathbb R$;
\item there exist a universal constant $C_{\mathrm{ent}}>0$ and a constant
$C>0$, independent of $N$, $\eta$, and $\delta t$ and depending only on
$B$, $\Omega$, $F$, $\lambda$, $c_*$, and the fixed basis, such that, with
probability at least $1-\eta$, simultaneously for all admissible
$a,b,c,\gamma$,
\begin{equation}\label{eq:solver-generalization}
\|\phi_c-\phi^*\|_{H^2}+|r_c-r^*|
\leq C\delta t^{-1}
\sqrt{\mathcal J_N(c;a,b,\gamma)+\Xi_{N,\eta}}.
\end{equation}
\end{enumerate}
Here
\begin{equation}\label{eq:explicit-generalization-gap}
\begin{aligned}
\Xi_{N,\eta}:={}&C_{\mathrm{ent}}G_B^2
\left[1+\sqrt{\log\left(1+
\frac{4\sqrt3\,B L_B}{G_B}\right)}\right]
\sqrt{\frac{3M+1}{N}}\\
&+2\sqrt2\,G_B^2\sqrt{\frac{\log(2/\eta)}{N}},
\end{aligned}
\end{equation}
where
\[
\begin{aligned}
K_A&:=(3+2\delta t\lambda)K_0+2\delta t K_2,
&R_B&:=3BK_0,\\
M_\ell&:=\sup_{|z|\leq R_B}|F^{(\ell)}(z)|,
&\ell&=1,2,\\
H_B&:=\frac{M_1}{\sqrt{c_*}},
&L_{H,B}&:=\frac{M_2}{\sqrt{c_*}}
+\frac{|\Omega|M_1^2}{2c_*^{3/2}},\\
G_B&:=BK_A+2\delta t\,BH_B+5BK_0>0.
\end{aligned}
\]
and
\[
L_B^2:=K_A^2+(2\delta t H_B)^2
+\bigl[(4\delta t B L_{H,B}+4)K_0\bigr]^2
+\bigl[(2\delta t B L_{H,B}+1)K_0\bigr]^2.
\]
\end{lemma}

The loss \(\mathcal J_N\) denotes the normalized residual used in the analysis. Moreover, up to a logarithmic factor,
\[
\Xi_{N,\eta}
=O\!\left(
G_B^2\sqrt{\frac{M+\log(1/\eta)}{N}}
\right).
\]
Here, the random basis is treated as fixed, and the probability is taken only over the collocation points. The proof is provided in Appendix~\ref{sec:proof solver error}.

Combining Lemmas~\ref{lem:error-semi-discrete}--\ref{lem:solver-error} gives
the following global estimate.

\begin{theorem}[Conditional high-probability global error estimate]
\label{thm:error}
Fix $T>0$, $0<\eta<1$, and
$0<\delta t\leq\min\{1,\delta t_0\}$, and set
$N_T:=\lfloor T/\delta t\rfloor\geq2$. Assume that:
\begin{enumerate}
\item the Allen--Cahn problem is posed on a domain covered by
Lemma~\ref{lem:solver-error};
\item the hypotheses of Lemma~\ref{lem:error-semi-discrete} and the
initialization hypotheses of Lemma~\ref{lem:error-accumulation} hold, with
$\mathcal E_{\rm st}\leq C_{\rm st}\delta t^4$ for a fixed $C_{\rm st}>0$;
\item the fixed random basis is boundary-conforming, the SAV integrals are
evaluated exactly, and one fixed set of i.i.d. uniform collocation points is
used at every BDF2 step;
\item the computed coefficients $c^k=(c_\phi^k,c_r^k)$ satisfy
$\|c^k\|_2\leq B$ for $0\leq k\leq N_T$.
\end{enumerate}
For $2\leq k\leq N_T$, let $\mathcal J_N^k$ be the loss
\eqref{eq:empirical-loss} of the subproblem producing
$(\hat\phi^k,\hat r^k)$. Then there exists an event $\mathcal A_{N,\eta}$ with
$\mathbb P(\mathcal A_{N,\eta})\geq1-\eta$ such that, on
$\mathcal A_{N,\eta}$, uniformly over all computed trajectories satisfying
the assumptions above,
\begin{equation}\label{eq:global-probability-bound}
\begin{aligned}
&\max_{1\leq n\leq N_T}
\left(\lambda\|\hat\phi^n-\phi(t_n)\|^2
+\|\nabla(\hat\phi^n-\phi(t_n))\|^2
+|\hat r^n-r(t_n)|^2\right)\\
&\qquad\leq C\left[
\delta t^4+\delta t^{-3}\sum_{k=2}^{N_T}
\bigl(\mathcal J_N^k+\Xi_{N,\eta}\bigr)\right].
\end{aligned}
\end{equation}
Moreover,
\[
\max_{2\leq k\leq N_T}
\bigl(\mathcal J_N^k+\Xi_{N,\eta}\bigr)
\leq C_{\rm tol}\delta t^8
\]
implies
\[
\max_{1\leq n\leq N_T}
\left(\|\hat\phi^n-\phi(t_n)\|_{H^1}
+|\hat r^n-r(t_n)|\right)
\leq C\delta t^2.
\]
The constant $C$ is independent of $N$, $\eta$, and $\delta t$.
\end{theorem}

\begin{remark}[Sampling requirement]\label{rem:sampling}
A sufficient condition for retaining the second-order accuracy is
\[
\max_{2\leq k\leq N_T}\mathcal J_N^k
\leq \frac12 C_{\rm tol}\delta t^8,
\qquad
\Xi_{N,\eta}\leq \frac12 C_{\rm tol}\delta t^8.
\]
Together with \eqref{eq:explicit-generalization-gap}, these conditions give an
explicit lower bound on the sample size \(N\). The resulting bound is generally
conservative because it is based on a worst-case uniform Rademacher estimate
over the entire coefficient class and all admissible histories. If only
first-order convergence is required, the tolerance may be relaxed, resulting
in a smaller lower bound on \(N\). Sharper sample-complexity estimates based on
localized complexity or additional stability properties are left for future
work.
\end{remark}

\begin{remark}[Cahn--Hilliard extension]\label{rem:ch-extension}
The same strategy extends to the Cahn--Hilliard equation by replacing
the Allen--Cahn stability step with its $H^{-1}$ counterpart. To keep the
presentation focused, the analogous derivation is omitted and no separate
Cahn--Hilliard global error theorem is stated.
\end{remark}

\begin{remark}[Irregular domains]\label{rem:ir-domain}
The analysis applies to bounded domains, including perforated domains, provided that the
elliptic regularity and boundary compatibility needed for the \(H^2\) stability
estimates hold.
\end{remark}

\begin{remark}[Boundary and omitted errors]\label{rem:additional-errors}
When boundary conditions are imposed approximately, the resulting error can be
included in Lemma~\ref{lem:solver-error} through an appropriate boundary
residual, together with the corresponding trace-stability and sampling
estimates. Since Lemma~\ref{lem:error-accumulation} also assumes exact boundary
compatibility, its deterministic stability estimate must be augmented by a
boundary lifting or suitable boundary-work terms. If this extension controls
the local error in the same \(H^2(\Omega)\times\mathbb R\) norm, the subsequent
time-accumulation argument remains unchanged. Numerical quadrature and other
implementation errors are beyond the scope of this analysis.
\end{remark}

\section{Experiments}
\subsection{Examples and Setups}
In this work, we systematically evaluate the proposed method on the Allen--Cahn and Cahn--Hilliard equations. 
We design four test cases, including 1D problems, 2D problems on a regular domain, and 2D problems on an irregular domain and 3D problems on an irregular domain. The detailed setups of these target examples 
are summarized in Examples~\ref{ex:1}-\ref{ex:6} as follows. We employ a feedforward neural network with the inner-layer parameters fixed. 
Detailed configurations can be found in Appendix~\ref{sec:model}. In solving the least-squares problems, an 
$L^2$ penalty with coefficient $\mathrm{1e}\text{-20}$ is explicitly introduced to mitigate overfitting. By default, 
we measure accuracy using the relative $L^2$ error at the final time, if not specified otherwise. For brevity, we place the visualization results in the Appendix~\ref{sec:visualresults}. In this section, we mainly discuss the properties demonstrated by our method in these examples.

\begin{example}
\label{ex:1} We consider the following 1D Allen--Cahn equation
\begin{equation}
  \frac{\partial \phi}{\partial t} = \epsilon^2 \frac{\partial^2 \phi}{\partial x^2} - 5(\phi^3 - \phi), 
\end{equation}
subjected to periodic boundary condition with $x \in \Omega = [-1, 1]$, $\epsilon=0.01$, $T=1$ and $\phi_0=x^2\cos(\pi x)$.
\end{example}

\begin{example}\label{ex:2}
We consider the following 1D Cahn--Hilliard equation
\begin{equation}
  \frac{\partial \phi}{\partial t} = - \frac{\partial^4 \phi}{\partial x^4} + \frac{1}{\epsilon^2}\frac{\partial^2(\phi^3-\phi)}{\partial x^2},
\end{equation}
with periodic boundary condition, $\Omega=[0,2\pi]$, $\epsilon=0.1$, $T=0.04$ and $\phi_0=0.5\sin(x)$.
\end{example}

\begin{example}\label{ex:3}
We consider the following 2D Allen--Cahn equation
\begin{equation}
  \frac{\partial \phi}{\partial t} = (\epsilon^2 \Delta \phi - (\phi^3 - \phi)), 
\end{equation}
with periodic condition and the domain $\Omega = [0,2\pi]\times[0,2\pi]$.
We set $\epsilon=0.1$, $T=5$ and $\phi_0 = 0.05\sin(x)\sin(y)$.
\end{example}

\begin{example}\label{ex:4}
We consider the following 2D Cahn--Hilliard equation
\begin{equation}
  \frac{\partial \phi}{\partial t} = - \Delta^2 \phi + \frac{\Delta(\phi^3-\phi)}{\epsilon^2},
\end{equation}
with periodic boundary condition, $\Omega=[0,2\pi]\times[0,2\pi]$, $\epsilon=0.1$, $T=0.016$ and $\phi_0=0.05\sin(x)\sin(y)$.
\end{example}

\begin{example}\label{ex:5}
We consider the following 2D Allen--Cahn equation
\begin{equation}
  \frac{\partial \phi}{\partial t} = M(\epsilon^2 \Delta \phi - (\phi^3 - \phi)), 
\end{equation}
with homogeneous Neumann condition and the domain 
\begin{equation*}
  \Omega = [-1,1]\times[-1,1]\backslash \bigcup_{k=1}^{3}\{\vx\in \mathbb{R}:\|\vx - \vc_k\|_2 \le r_k \},
\end{equation*}
where $\vc_1=(0, 0), \vc_2=(0.6,0.6), \vc_3=(-0.5,0.6)$ and $r_1 = 0.4$, $r_2=r_3=0.2$. 
This domain denotes a square with three circular subregions removed. We set $M=50$, $\epsilon=0.1$, $T=0.2$,
the initial data $\phi_0 = 0.1\text{Rand}(x)$ where $\text{Rand}(x)$ stands for random values in $[-1, 1]$.
\end{example}

\begin{example}\label{ex:6}
We consider the following 2D Cahn--Hilliard equation
\begin{equation}
  \begin{aligned}
\frac{\partial \phi}{\partial t} &= - M \Delta \mu \\
\mu &= -\Delta \phi + \frac{(\phi^3-\phi)}{\epsilon^2},
    \end{aligned}
\end{equation}
with homogeneous Neumann condition, $M=200$, $\epsilon=0.1$,  $\phi_0=0.05\sin(x)\sin(y)$ , $T=0.02$ and the domain
\begin{equation*}
  \Omega = [-1,1]\times[-1,1]\backslash \{\vx\in \mathbb{R}:\|\vx \|_2 \le 0.4 \}.
\end{equation*}
\end{example}

\begin{example}\label{ex:7}
    We consider the following 3D Allen-Cahn equation
    \begin{equation} \phi_t = M(\epsilon^2 \Delta \phi - (\phi^3-\phi) ),
    \end{equation}
    with homogeneous Neumann condition, $M = 100$, $\epsilon^2 = 0.001$, $\phi_0 = \frac{\tanh(\sqrt{2}(|\boldsymbol{x}-\boldsymbol{r}|-0.4))}{2}$, $\vr = (0.25,0,0.15)$ and the domain
    $$
    \Omega = [-1,1]^3\backslash\{ \vx\in \mathbb{R}:\|\vx-\vr\|\le 0.35 \} 
    $$
\end{example}

\subsection{Relaxation}
The Cahn--Hilliard equation requires the computation of fourth-order spatial derivatives, which poses additional 
challenges for automatic differentiation. In practice, generating the Jacobian matrix 
for fourth-order equations can be significantly more expensive than for second-order problems, 
while the presence of high-order derivatives also makes the Jacobian more susceptible to 
numerical perturbations. Thus, we introduce an auxiliary network to learn $\mu^{n+1}$, then 
we need to solve $\phi^{n+1},\mu^{n+1}$ and $r^{n+1}$ at each step.
\subsection{Results and Analysis}
\textbf{Performance.} While our main objective is to develop an energy-stable neural network method, we also include experiments on Example~\ref{ex:1} which is a typical benchmark, to facilitate comparisons with existing approaches. We categorize the compared methods into two groups: 
PINN-based approaches and their variants (left column), and a second group (right column) consisting of 
attention-based models, the state-of-the-art PINN method PirateNet~\cite{WangPirateNets2024JMLR}, 
classical numerical method (Fourier Spectral method), 
the SDTM~\cite{WangSemidiscrete2025} approach, and our proposed method. To ensure a fair comparison, we compare the global relative $L^2$ error for $t\in[0,1]$. 
We employ the semi-implicit BDF2 scheme with 
$\delta t= 1\mathrm{e}\text{-3}$ across all three iterative methods (Fourier, SDTM, and ours), and the results of Vanilla~\cite{RaissiPhysicsinformed2019JCP},
Adaptive~\cite{WightSolving2021CCP}, Time marching~\cite{MatteyNovel2022CMAMEa}, Causal~\cite{WangRespecting2024CMAME}, RBA~\cite{AnagnostopoulosResidualbased2024CMAME},
JAX-PI~\cite{WangUnderstanding2021SJSC} and Self-Attention~\cite{MCCLENNY2023111722} are referred to Ref.~\cite{WangPirateNets2024JMLR}. As shown in Table~\ref{tab:ac1d}, the Fourier spectral 
method yields the best performance. As the time step size is reduced, our method achieves performance comparable to the Fourier spectral method, while 
outperforming the state-of-the-art PINN methods. While SDTM exhibits higher accuracy than our approach in this case, this suggests that 
incorporating the SAV framework may lead to a trade-off between accuracy and stability.
In contrast, due to the lack of stability guarantees, SDTM performs noticeably worse on the Cahn--Hilliard equation, 
where energy stability plays a critical role, which will be observed in Tab.~\ref{tab:ablation}.

\begin{table}[h]
  \caption{Performance of Example~\ref{ex:1}: Global relative $L^2$ error, our method was evaluated over five independent runs}
  \label{tab:ac1d}
  \begin{center}
    \begin{small}
        \begin{tabular}{lcccr}
          \toprule
          Method(PINN)  & Rel. Err. & Method  & Rel. Err. \\
          \midrule
          Vanilla   & $4.98\mathrm{E}\text{-01}$ &  Self-attention   &  $2.10\mathrm{E}\text{-02}$\\
          Adaptive & $2.33\mathrm{E}\text{-02}$ & \textbf{PirateNet} & $\mathbf{2.24}\mathbf{E}\textbf{-05}$ \\
          Time marching & $1.68\mathrm{E}\text{-02}$ &  Fourier &  $\mathbf{1.23\mathbf{E}\textbf{-05}}$ \\
          Causal    & $6.29\mathrm{E}\text{-05}$ &  SDTM    &  $1.33\mathrm{E}\text{-05}$ \\
           RBA & $4.55\mathrm{E}\text{-05}$ & \textbf{Ours}  & $6.12\mathrm{E}\text{-05}(3.81\mathrm{E}\text{-06})$ \\
          JAX-PI     &  $5.37\mathrm{E}\text{-05}$ &  ($\delta t=5\mathrm{e}\text{-4}$)  & $\mathbf{2.09\mathbf{E}\textbf{-05}}(4.68\mathrm{E}\text{-06})$ \\
          \bottomrule
        \end{tabular}
    \end{small}
  \end{center}
  \vskip -0.2in
\end{table}

\noindent\textbf{Stability and Convergence Rate.}
As described in Theorem~\ref{thm:energy-stability}, the scheme is unconditionally energy stable. We test all the examples
about their modified energy evolution in the iterations. All examples are presented with the time interval normalized to $[0,1]$. 
Fig.~\ref{fig:stable}(a) indicates the proposed method satisfies
energy stability for all examples, while the SDTM approach blows up for the Cahn--Hilliard equations. For examples posed 
on irregular domains, evaluating the reference energy would involve additional mesh-based differentiation, and is thus omitted.
Nonetheless, the results computed by our method consistently exhibit energy-stable behavior. Furthermore, we test 
the convergence rate of Example~\ref{ex:3} up to $T=1$ using BDF2
scheme and Euler method in Fig.~\ref{fig:stable}(b). These results confirm the convergence of the proposed method, 
as established in Theorem~\ref{thm:error}.

\begin{figure}[t]
  % \vskip 0.2in
  \begin{center}
    \centerline{\includegraphics[width=0.98\columnwidth]{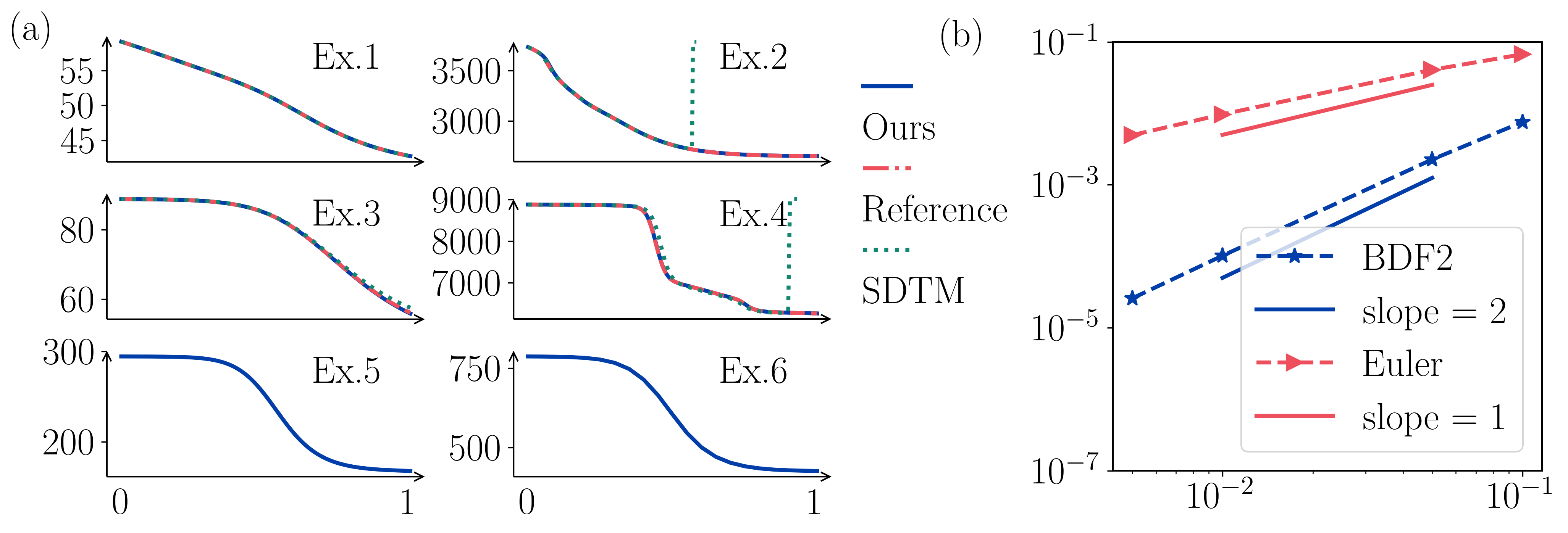}}
    \caption{
      (a) Energy evolution for Example~1--6: all examples exhibit energy dissipation. 
      For Examples~1--4, the proposed method closely follows the reference energy, while the SDTM method becomes unstable for the Cahn--Hilliard equation; 
      (b) Convergence behavior for Example~\ref{ex:3} using 
      $\delta t = 0.1, 0.05, 0.01, 0.005$: the BDF2 exhibits second-order convergence, while the Euler method exhibits first-order convergence.
    }
  \label{fig:stable}
  \end{center}
  \vskip -0.2in
\end{figure}

\noindent\textbf{Irregular Domain.}
Our method benefits from the flexibility of mesh-free neural networks, 
allowing it to handle PDEs posed on irregular domains in a natural and unified manner.
This contrasts with mesh-based methods, i.e.,\ finite element method, which often 
require nontrivial mesh generation and refinement for complex geometries. We show the result
of Example~\ref{ex:5} in Fig.~\ref{fig:ac2dp}, where the first line is the reference solution and the second line is our method using BDF2 scheme. Since smooth 
neural network basis functions have difficulty representing random initial conditions, we start the 
simulation from $t=0.02$. From the figure, 
we observe that our solution closely matches the reference solution, with no visible discrepancies.
As we consider the isotropic setting, the contact angle at the central circular boundary is expected to be perpendicular, which is 
indeed observed and marked by the red squares. This agreement demonstrates that the proposed method 
preserves the underlying physical behavior of the system. The relative $L_2$ error at the final time is approximately 
$7.32\mathrm{E}\text{-03}$. Moreover, we show the results of Example~\ref{ex:6} in Fig.\ref{fig:ex6}.

\begin{table}[h]
  \caption{Parameter sensitivity analysis for Example~\ref{ex:2} }
  \centering
  \begin{tabular}{lccr}
    \toprule
      & $w_r = 1$ & $w_r=50$ &$w_r=1000$ \\
    \midrule
    $R_\text{init} = 1$ & $9.86\mathrm{E}\text{-01}$ & $4.96\mathrm{E}\text{-01}$ & $2.64\mathrm{E}\text{-01}$ \\
    &($2.53\mathrm{E}\text{-01}$)&($5.80\mathrm{E}\text{-02}$)&($1.95\mathrm{E}\text{-02}$)\\[2pt]
    % \vskip 0.1in
    $R_\text{init} = 2$ & $2.90\mathrm{E}\text{-01}$ & $9.96\mathrm{E}\text{-02}$ & $1.30\mathrm{E}\text{-02}$ \\
        &($3.39\mathrm{E}\text{-02}$)&($7.10\mathrm{E}\text{-03}$)&($7.41\mathrm{E}\text{-04}$)\\[2pt]

    $R_\text{init} = 3$ & $5.86\mathrm{E}\text{-02}$ & $2.41\mathrm{E}\text{-03}$ & $2.40\mathrm{E}\text{-03}$ \\
        &($5.06\mathrm{E}\text{-04}$)&($5.06\mathrm{E}\text{-04}$)&($5.06\mathrm{E}\text{-04}$)\\[2pt]

    % $R_\text{i} = 5$ & $2.11\mathrm{E}\text{-03}$ & $2.11\mathrm{E}\text{-03}$ & $2.11\mathrm{E}\text{-03}$ \\
    %     &($5.06\mathrm{E}\text{-04}$)&($5.06\mathrm{E}\text{-04}$)&($5.06\mathrm{E}\text{-04}$)\\
  \end{tabular}
  \label{tab:w_and_r}
\end{table}

\noindent\textbf{Sensitivity Analysis and Ablation.} 
In the implementation of the proposed method, we jointly optimize the two terms 
in Eq.~\eqref{eq:sys2sol}, which correspond to the approximation errors of the phase variable $\phi$
and the auxiliary variable $r$, respectively. A weighting parameter $w_r$ is 
introduced for the loss term associated with $r$. For the neural network, the inner-layer parameters
are randomly initialized with uniform distribution $\mathcal{U}[-R_\text{init},R_\text{init}]$. We investigate the impact of 
these factors on the results in the following experiments. 
In addition, for the Cahn--Hilliard equation, we investigate the effect of relaxation as well as the effect
of SAV framework.

\begin{figure}[t]
  % \vskip 0.2in
  \centering
    \includegraphics[width=0.98\columnwidth]{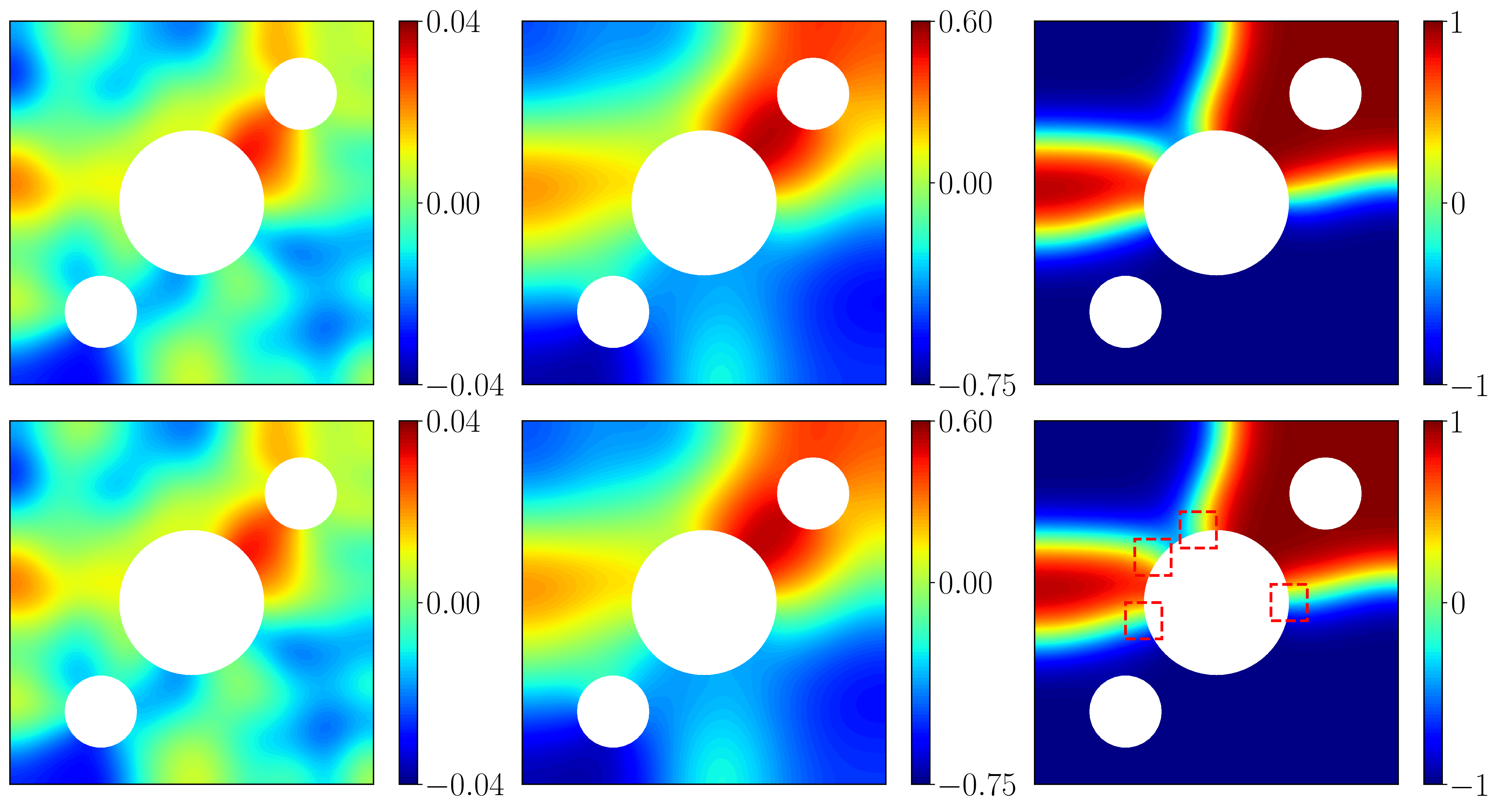}
    \caption{
    Evolution of Example~\ref{ex:5} for $t=0.02$, $t=0.1$, $t=0.2$. 
    Top row: reference solution; bottom row: predicted solution.
    }
  \label{fig:ac2dp}
\end{figure}

From Tab.\ref{tab:w_and_r}, we observe that the experimental performance improves as 
$w_r$ and $R_\text{init}$ increase. Increasing $w_r$ is beneficial as it helps enforce 
energy stability during the iterative solution process. In contrast, the pronounced 
performance differences across different values of $R_\text{init}$ indicate that a more 
appropriate initialization (a better basis) is required to better approximate the solution of the 
equation.

\begin{table}[h]
  \caption{Ablation for relaxation and SAV formulation, ``-" indicates SDTM blows up in the iterations.}
  \centering
  \begin{tabular}{lcccr}
    \toprule
     Formulation & \multicolumn{2}{c} {SAV} & SDTM \\
     Relaxation & w/ & w/o &\\ 
    \midrule
    Example~\ref{ex:2}& $2.65\mathrm{E}\text{-04}$ & $2.11\mathrm{E}\text{-03}$ & \multicolumn{1}{|c} - \\ 
    Example~\ref{ex:4}& $1.92\mathrm{E}\text{-02}$ & $6.54\mathrm{E}\text{-01}$ & \multicolumn{1}{|c} - \\
  \end{tabular}
  \label{tab:ablation}
\end{table}

\section{Conclusion}
In this work, we propose an energy-stable neural network-based method for solving the Allen--Cahn and Cahn--Hilliard equations. By integrating classical time discretization with 
mesh-free neural network representation, the proposed method explicitly preserves discrete energy dissipation.
We established rigorous stability and error analyses, providing theoretical guarantees for the proposed method. 
Numerical experiments further demonstrated that our approach consistently exhibits energy-stable behavior, 
achieves the convergence order predicted by the theory, and outperforms existing neural network-based solvers while 
remaining competitive with classical numerical schemes. In addition, the method naturally extends to irregular computational domains, 
highlighting the flexibility of neural network representations.

%\begin{acknowledgements}
%\end{acknowledgements}



\bibliography{ref}
\bibliographystyle{spmpsci}
\nocite{WuStabilized2014IJNMBE}
\nocite{yang2016linear}
\nocite{XuOverview2025CAMC}
\nocite{SunSurrogate2020CMAME}
\nocite{BarronUniversal1993ITIT}
\nocite{SunLocal2024JCAM}
\nocite{YehudaiPower2019ANIPS}
\nocite{ChenDuality2023}

\newpage
\appendix
\onecolumn
\section{Appendix}
\subsection{Proof}
We collect here the proofs of all lemmas and theorems used in this paper that are not quoted verbatim from the literature. Before presenting the proofs, we first fix the notation and conventions that will be used throughout this section.

\textbf{Notation.}
We use \(\|\cdot\|\) to denote the \(\ell^2\)-norm of vectors and the \(L^2\)-norm of functions, and \(\langle \cdot, \cdot \rangle\) to denote the \(L^2\) inner product of two functions. We denote the spaces \(L^{p}(\Omega)\) by \(L^{p}\) for short. The Sobolev spaces \(H^{s}\) with non-integer order \(s\) will also be used.
In addition, we use \(C\) to denote a generic positive constant that may vary from line to line but is always independent of the time step size \(\delta t\).
We use $E_1[v]=\int_\Omega F(v)\D x+C_0$ and assume
$E_1[v]\geq c_*>0$ for all states under consideration.

\subsubsection{$H^2$ bound for the semi-discrete scheme~\eqref{eq:savbdf2}}\label{sec:proof H2}
For the Allen--Cahn equation, the modified energy law gives uniform $H^1$
and SAV bounds for the exact semi-discrete solution. The following standard
estimate supplies the additional spatial regularity used below.
\begin{lemma}[$H^2$ bound for the semi-discrete scheme~\eqref{eq:savbdf2}]
\label{lem:H2-estimate}
Let $d\leq3$, and let $\Omega$ be a periodic box or a smooth bounded domain
with homogeneous Neumann boundary conditions. Fix $T>0$ and
$0<\delta t\leq1$. Under Assumption~\ref{assump:F-smooth}, suppose that
$\phi^0\in H^3(\Omega)$ satisfies the boundary compatibility conditions.
Let $\{(\phi^n,r^n)\}$ be the exact BDF2--SAV sequence initialized by one
exact Euler--SAV step. Then
\begin{equation}\label{eq:H2-L2}
\max_{0\leq n\leq \lfloor T/\delta t\rfloor}
\|\Delta\phi^n\|^2
\leq C\left(T+1+\|\Delta\phi^0\|^2
+\delta t\|\nabla\Delta\phi^0\|^2\right),
\end{equation}
where $C$ is independent of $\delta t$. Consequently,
$\max_n\|\phi^n\|_{H^2}\leq C_T$.
\end{lemma}
\begin{proof}
The Euler and BDF2 energy laws imply
\begin{equation}\label{eq:H2-energy-bound}
\sup_{0\leq n\leq \lfloor T/\delta t\rfloor}
\bigl(\|\phi^n\|_{H^1}+|r^n|\bigr)\leq C.
\end{equation}
The integrations below can be justified by the standard Galerkin
approximation. Lemma~2.3 of Ref.~\cite{shen2018convergence}, together with
\eqref{eq:H2-energy-bound} and $E_1\geq c_*$, gives, for every
$\varepsilon>0$,
\begin{equation}\label{eq:H2-nonlinear-bound}
C\|\nabla F'(v)\|^2
\leq \varepsilon\|\nabla\Delta v\|^2+C_\varepsilon
\end{equation}
for $v=\phi^0$ and $v=2\phi^n-\phi^{n-1}$, $n\geq1$.

We first estimate the Euler--SAV start. Eliminating the chemical potential
from~\eqref{eq:saveuler}, testing the resulting equation with
$\Delta^2\phi^1$, and using~\eqref{eq:H2-nonlinear-bound} yield
\begin{equation}\label{eq:H2-Euler-start}
\|\Delta\phi^1\|^2
+\delta t\|\nabla\Delta\phi^1\|^2
\leq C\left(1+\|\Delta\phi^0\|^2
+\delta t\|\nabla\Delta\phi^0\|^2\right).
\end{equation}
Indeed, the time-difference term is
\[
\left\langle\phi^1-\phi^0,\Delta^2\phi^1\right\rangle
=\frac12\left(
\|\Delta\phi^1\|^2-\|\Delta\phi^0\|^2
+\|\Delta(\phi^1-\phi^0)\|^2\right),
\]
whereas the nonlinear term is bounded by
\[
\left|r^1\langle h(\phi^0),\Delta^2\phi^1\rangle\right|
\leq \frac12\|\nabla\Delta\phi^1\|^2
+C\|\nabla F'(\phi^0)\|^2.
\]

For the BDF2 steps, set
\[
y^n:=\Delta\phi^n,\qquad z^n:=\nabla\Delta\phi^n,
\]
and introduce the positive-definite two-step quantity
\[
\mathcal Q^n:=\frac52\|y^n\|^2+\frac12\|y^{n-1}\|^2
-2\langle y^n,y^{n-1}\rangle,\qquad n\geq1.
\]
In particular, $\mathcal Q^n\geq\frac12\|y^n\|^2$ and
$\mathcal Q^1\leq C(\|y^1\|^2+\|y^0\|^2)$. Testing the first equation of
\eqref{eq:savbdf2}, after eliminating $\mu^{n+1}$, with
$\Delta^2\phi^{n+1}$ gives
\begin{align*}
&\frac{1}{2\delta t}
\left\langle3\phi^{n+1}-4\phi^n+\phi^{n-1},
\Delta^2\phi^{n+1}\right\rangle
+\|z^{n+1}\|^2+\lambda\|y^{n+1}\|^2\\
&\qquad=-r^{n+1}
\left\langle h(2\phi^n-\phi^{n-1}),
\Delta^2\phi^{n+1}\right\rangle.
\end{align*}
The BDF2 identity
\begin{align*}
&\left\langle3\phi^{n+1}-4\phi^n+\phi^{n-1},
\Delta^2\phi^{n+1}\right\rangle\\
&\qquad=\mathcal Q^{n+1}-\mathcal Q^n
+\frac12\|y^{n+1}-2y^n+y^{n-1}\|^2.
\end{align*}
The nonlinear term is bounded as in the Euler step by
\[
\frac12\|z^{n+1}\|^2
+C\|\nabla F'(2\phi^n-\phi^{n-1})\|^2.
\]
Thus, choosing $\varepsilon=1/20$ in~\eqref{eq:H2-nonlinear-bound}
and using the BDF2 identity above imply
\begin{equation}\label{eq:H2-BDF-recurrence}
\mathcal Q^{n+1}-\mathcal Q^n
+\delta t\|z^{n+1}\|^2
\leq \frac45\delta t\|z^n\|^2
+\frac15\delta t\|z^{n-1}\|^2+C\delta t,
\end{equation}
where we used
$\|2z^n-z^{n-1}\|^2\leq8\|z^n\|^2+2\|z^{n-1}\|^2$.

For $m\geq2$, summing~\eqref{eq:H2-BDF-recurrence} from $n=1$ to
$m-1$ gives the endpoint estimate
\begin{equation}\label{eq:H2-BDF-sum}
\mathcal Q^m+\delta t\|z^m\|^2
+\frac15\delta t\|z^{m-1}\|^2
\leq \mathcal Q^1+\delta t\|z^1\|^2
+\frac15\delta t\|z^0\|^2+CT.
\end{equation}
Combining~\eqref{eq:H2-BDF-sum} with the start-up estimate
\eqref{eq:H2-Euler-start} and
$\mathcal Q^m\geq\frac12\|\Delta\phi^m\|^2$ proves
\eqref{eq:H2-L2} for $m\geq2$; the cases $m=0,1$ follow directly from
the initial data and~\eqref{eq:H2-Euler-start}. The $H^2$ bound follows
from elliptic regularity and~\eqref{eq:H2-energy-bound}.
\end{proof}

\subsubsection{Proof of Lemma~\texorpdfstring{\ref{lem:error-accumulation}}{Lemma \ref{lem:error-accumulation}}}\label{sec:proof error accumulation}
\begin{proof}
Set $\mathcal A:=-\Delta+\lambda I$ and
$H(v):=F'(v)/\sqrt{E_1[v]}$. For $n\geq1$, write
\[
\bar{\hat\phi}^{\,n+1}:=2\hat\phi^n-\hat\phi^{n-1},
\qquad
\bar\phi^{n+1}:=2\phi^n-\phi^{n-1},
\]
and define $\hat w^{n+1}:=\hat\mu^{n+1}-\mu^{n+1}$, where
$\hat\mu^{n+1}=\mathcal A\hat\phi^{n+1}
+\hat r^{n+1}H(\bar{\hat\phi}^{\,n+1})$.

Let $\eta_\phi^{n+1}:=\hat\phi^{n+1}-\hat\phi^{n+1,*}$ and
$\eta_r^{n+1}:=\hat r^{n+1}-\hat r^{n+1,*}$. Subtracting the exact
BDF2--SAV equations from the linear subproblem formed with the numerical
history gives
\begin{align}
\delta_2\hat e^{n+1}
&=-2\delta t\,\hat w^{n+1}+\xi^{n+1},
\label{eq:err_phi}\\
\delta_2\hat s^{n+1}
&=\frac12\left\langle H(\bar{\hat\phi}^{\,n+1}),
\delta_2\hat e^{n+1}\right\rangle
+\mathcal N^{n+1}+\tau^{n+1},
\label{eq:err_s}\\
\hat w^{n+1}
&=\mathcal A\hat e^{n+1}
+\hat s^{n+1}H(\bar{\hat\phi}^{\,n+1})
+\mathcal R^{n+1},
\label{eq:err_mu}
\end{align}
where $\delta_2v^{n+1}:=3v^{n+1}-4v^n+v^{n-1}$ and
\begin{align*}
\xi^{n+1}
&:=3\eta_\phi^{n+1}
+2\delta t\bigl(\mathcal A\eta_\phi^{n+1}
+\eta_r^{n+1}H(\bar{\hat\phi}^{\,n+1})\bigr),\\
\tau^{n+1}
&:=3\eta_r^{n+1}
-\frac32\left\langle H(\bar{\hat\phi}^{\,n+1}),
\eta_\phi^{n+1}\right\rangle,\\
\mathcal N^{n+1}
&:=\frac12\left\langle
H(\bar{\hat\phi}^{\,n+1})-H(\bar\phi^{n+1}),
\delta_2\phi^{n+1}\right\rangle,\\
\mathcal R^{n+1}
&:=r^{n+1}\bigl[H(\bar{\hat\phi}^{\,n+1})
-H(\bar\phi^{n+1})\bigr].
\end{align*}

Lemma~\ref{lem:H2-estimate} bounds the exact semi-discrete sequence, whereas
the numerical sequence is bounded by the hypothesis of the present lemma.
Together with $H^2(\Omega)\hookrightarrow L^\infty(\Omega)$ for $d\leq3$
and the lower bound on $E_1$, these estimates imply
\begin{align}
\|\xi^{n+1}\|+|\tau^{n+1}|
&\leq C\epsilon_{n+1},
\label{eq:local-defect-bound}\\
\|H(\bar{\hat\phi}^{\,n+1})-H(\bar\phi^{n+1})\|
&\leq C\bigl(\|\hat e^n\|+\|\hat e^{n-1}\|\bigr).
\label{eq:lipschitz_bound}
\end{align}
In particular, this use of Lemma~\ref{lem:H2-estimate} concerns only the exact
semi-discrete solution and does not assert an a priori bound for the numerical
sequence.

Taking the inner products of \eqref{eq:err_phi} and \eqref{eq:err_mu} with
$\hat w^{n+1}$ and $\delta_2\hat e^{n+1}$, respectively, and multiplying
\eqref{eq:err_s} by $2\hat s^{n+1}$ yields
\begin{equation}\label{eq:energy_identity}
\begin{aligned}
&\langle\mathcal A\hat e^{n+1},\delta_2\hat e^{n+1}\rangle
+2\hat s^{n+1}\delta_2\hat s^{n+1}
+2\delta t\|\hat w^{n+1}\|^2\\
&\quad=\langle\xi^{n+1},\hat w^{n+1}\rangle
-\langle\mathcal R^{n+1},\delta_2\hat e^{n+1}\rangle
+2\hat s^{n+1}\mathcal N^{n+1}
+2\hat s^{n+1}\tau^{n+1}.
\end{aligned}
\end{equation}
Using \eqref{eq:err_phi} in the second term on the right, followed by
\eqref{eq:local-defect-bound}, \eqref{eq:lipschitz_bound}, Young's inequality,
and the uniform $L^2$ bound for $\mu^{n+1}$, gives
\begin{equation}\label{eq:energy-rhs-bound}
\mathrm{RHS~of~\eqref{eq:energy_identity}}
\leq\delta t\|\hat w^{n+1}\|^2
+C\delta t\bigl(\|\hat e^n\|^2+\|\hat e^{n-1}\|^2
+|\hat s^{n+1}|^2\bigr)
+C\delta t^{-1}\epsilon_{n+1}^2.
\end{equation}

For $\|v\|_{\mathcal A}^2:=\langle\mathcal Av,v\rangle$, introduce
\[
\mathcal E_e^n:=\frac12\left(
\|\hat e^n\|_{\mathcal A}^2
+\|2\hat e^n-\hat e^{n-1}\|_{\mathcal A}^2\right)
+|\hat s^n|^2+|2\hat s^n-\hat s^{n-1}|^2.
\]
The exact BDF2 identity
\[
2a(3a-4b+c)=a^2+(2a-b)^2-b^2-(2b-c)^2+(a-2b+c)^2
\]
applied to \eqref{eq:energy_identity}--\eqref{eq:energy-rhs-bound} implies
\[
\mathcal E_e^{n+1}-\mathcal E_e^n
\leq C\delta t\bigl(\mathcal E_e^{n+1}+\mathcal E_e^n\bigr)
+C\delta t^{-1}\epsilon_{n+1}^2.
\]
For $\delta t\leq\delta t_0$, the first term on the right is absorbed, and
the discrete Gronwall inequality gives
\[
\mathcal E_e^m\leq
C\left(\mathcal E_e^1+\delta t^{-1}\sum_{j=2}^{m}\epsilon_j^2\right).
\]
Since the initial data agree, $\mathcal E_e^1\leq C\mathcal E_{\rm st}$. Finally,
$\|\hat e^m\|_{\mathcal A}^2+|\hat s^m|^2\leq2\mathcal E_e^m$, which proves
\eqref{eq:inexact-solve-stability}.
\end{proof}

\subsubsection{Proof of Lemma~\texorpdfstring{\ref{lem:solver-error}}{Lemma \ref{lem:solver-error}}}\label{sec:proof solver error}
\begin{proof}
Fix admissible $a$ and $b$, abbreviate $\psi=\psi_{a,b}$, and define
$A=(A_1,A_2):H^2_{\mathcal B}(\Omega)\times\mathbb R
\to L^2(\Omega)\times\mathbb R$ by
\begin{equation}\label{eq:ope}
\begin{aligned}
A_1(\phi,r)&:=\mathcal L_{\delta t}\phi+2\delta t\,r\psi,\\
A_2(\phi,r)&:=r-\tfrac12\langle\psi,\phi\rangle.
\end{aligned}
\end{equation}
We first establish solvability and a uniform stability estimate for $A$.
Given $(g,\zeta)\in L^2(\Omega)\times\mathbb R$, the equation
$A_2(\phi,r)=\zeta$ gives
$r=\zeta+\frac12\langle\psi,\phi\rangle$. Substitution into
$A_1(\phi,r)=g$ yields the variational problem
\[
\begin{aligned}
&(3+2\delta t\lambda)(\phi,v)
+2\delta t(\nabla\phi,\nabla v)
+\delta t\langle\psi,\phi\rangle\langle\psi,v\rangle\\
&\qquad=(g-2\delta t\,\zeta\psi,v)
\end{aligned}
\]
for all admissible $v\in H^1(\Omega)$. The left-hand side is continuous and
coercive. Hence Lax--Milgram gives a unique $H^1$ solution, and the assumed
elliptic regularity gives $\phi\in H^2_{\mathcal B}(\Omega)$. Thus $A$ is
bijective.

To make the stability constant uniform in $\delta t$, set
$g=A_1(\phi,r)$ and $\zeta=A_2(\phi,r)$. Testing the variational equation above
with $v=\phi$ and using $0<\delta t\leq1$ and $\|\psi\|\leq C$ gives
\[
3\|\phi\|^2
\leq\bigl(\|g\|+2\delta t\,|\zeta|\|\psi\|\bigr)\|\phi\|,
\]
and hence $\|\phi\|\leq C(\|g\|+|\zeta|)$. The identity
$r=\zeta+\frac12\langle\psi,\phi\rangle$ then gives the same bound for $|r|$.
Consequently,
\begin{equation}\label{eq:phi-r-bound}
\|\phi\|+|r|
\leq C\bigl(\|A_1(\phi,r)\|+|A_2(\phi,r)|\bigr).
\end{equation}
Moreover,
\[
2\delta t\,\Delta\phi
=(3+2\delta t\lambda)\phi+2\delta t\,r\psi-A_1(\phi,r).
\]
Combining this identity with \eqref{eq:phi-r-bound} and the elliptic
$H^2$ estimate yields
\begin{equation}\label{eq:inequality}
\|\phi\|_{H^2}+|r|
\leq C\delta t^{-1}
\bigl(\|A_1(\phi,r)\|+|A_2(\phi,r)|\bigr).
\end{equation}

It remains to control the continuous residual by its sampled values. Let
$\theta$ collect the coefficient blocks $c_\phi,c_r,a,b$. Then
$\theta\in\mathbb R^{3M+1}$, and the coefficient constraints imply
$\|\theta\|_2\leq\sqrt3B$. The definitions of $H_B$ and $L_{H,B}$ give
\[
\begin{aligned}
\|h(v)\|_{L^\infty}&\leq H_B,\\
\|h(v)-h(w)\|_{L^\infty}
&\leq L_{H,B}\|v-w\|_{L^\infty}.
\end{aligned}
\]
for states induced by the coefficient ball. Consequently,
\begin{equation}\label{eq:residual-envelope-lipschitz}
\|\mathcal R_\theta\|_{L^\infty}\leq G_B,
\qquad
\|\mathcal R_\theta-\mathcal R_{\widetilde\theta}\|_{L^\infty}
\leq L_B\|\theta-\widetilde\theta\|_2.
\end{equation}
The second estimate follows by separating the current phase coefficient, the
current scalar coefficient, and the two history coefficients; the four
corresponding constants are precisely the four terms in the definition of
$L_B^2$.

Let $\mathcal R_B$ be the class of residual functions generated by the
admissible parameters. Since
$|q^2-\widetilde q^{\,2}|\leq2G_B|q-\widetilde q|$, the Euclidean covering
number of $\mathcal R_B^2:=\{q^2:q\in\mathcal R_B\}$ satisfies
\[
\log\mathcal N\left(\varepsilon,\mathcal R_B^2,L^\infty\right)
\leq(3M+1)\log\left(1+
\frac{4\sqrt3\,B G_BL_B}{\varepsilon}\right).
\]
The contribution of one fixed anchor is $O(G_B^2/\sqrt N)$. The standard
Dudley entropy estimate therefore gives
\[
\mathfrak R_N(\mathcal R_B^2)
\leq C G_B^2
\left[1+\sqrt{\log\left(1+
\frac{4\sqrt3\,B L_B}{G_B}\right)}\right]
\sqrt{\frac{3M+1}{N}}.
\]
Here $\mathfrak R_N$ denotes the expected Rademacher complexity; the anchor
term is absorbed into the constant on the right.
The residual class is separable because it is continuously parameterized by
a compact finite-dimensional coefficient set, so the required supremum is
measurable.
Combining this estimate with the uniform Rademacher bound of
Ref.~\cite{shin2023error} yields, with probability at least $1-\eta$,
\begin{equation}\label{eq:continuous-empirical-residual}
|\Omega|^{-1}\|\mathcal R_c\|^2
\leq\frac1N\sum_{i=1}^N|\mathcal R_c(X_i)|^2+\Xi_{N,\eta}
\end{equation}
simultaneously over the coefficient ball. The factor $|\Omega|$ is fixed and
is absorbed into the constant below. Applying \eqref{eq:inequality} to
$(\phi_c-\phi^*,r_c-r^*)$, whose two operator residuals are exactly the two
terms in \eqref{eq:empirical-loss}, and then using
\eqref{eq:continuous-empirical-residual} proves
\eqref{eq:solver-generalization}.
\end{proof}

\subsubsection{Proof of Theorem~\texorpdfstring{\ref{thm:error}}{Theorem \ref{thm:error}}}\label{sec:proof main theorem}
\begin{proof}
Let $(\phi^n,r^n)$ be the exact semi-discrete BDF2--SAV sequence. By
Lemma~\ref{lem:error-semi-discrete},
\begin{equation}\label{eq:truncation_error}
\max_{1\leq n\leq N_T}
\left(\lambda\|\phi^n-\phi(t_n)\|^2
+\|\nabla(\phi^n-\phi(t_n))\|^2
+|r^n-r(t_n)|^2\right)
\leq C\delta t^4.
\end{equation}

Let $\mathcal A_{N,\eta}$ be the uniform sampling event in
Lemma~\ref{lem:solver-error}. Because that lemma is uniform over the current
and history coefficient balls, the same event applies at every time level,
even though the computed history depends on the collocation points. On
$\mathcal A_{N,\eta}$, its parameters at step $k$ are
\[
a=c_\phi^{k-1},\qquad b=c_\phi^{k-2},\qquad
\gamma=\frac{4\hat r^{k-1}-\hat r^{k-2}}{3},\qquad c=c^k.
\]
Thus, for $2\leq k\leq N_T$,
\begin{equation}\label{eq:solver_eps}
\epsilon_k^2
\leq C\delta t^{-2}
\bigl(\mathcal J_N^k+\Xi_{N,\eta}\bigr).
\end{equation}
Substitution into Lemma~\ref{lem:error-accumulation}, together with
$\mathcal E_{\rm st}\leq C\delta t^4$, gives
\[
\begin{aligned}
&\max_{1\leq n\leq N_T}
\left(\lambda\|\hat\phi^n-\phi^n\|^2
+\|\nabla(\hat\phi^n-\phi^n)\|^2
+|\hat r^n-r^n|^2\right)\\
&\qquad\leq C\left[\delta t^4+\delta t^{-3}
\sum_{k=2}^{N_T}(\mathcal J_N^k+\Xi_{N,\eta})\right].
\end{aligned}
\]
Combining this estimate with \eqref{eq:truncation_error} and using
$(a+b)^2\leq2a^2+2b^2$ proves
\eqref{eq:global-probability-bound}. Finally, if each summand is bounded by
$C_{\rm tol}\delta t^8$, then $N_T\leq T/\delta t$ implies
\[
\delta t^{-3}\sum_{k=2}^{N_T}
(\mathcal J_N^k+\Xi_{N,\eta})\leq C\delta t^4.
\]
Taking square roots and using $\lambda>0$ proves the final assertion.
\end{proof}

\subsection{Experiments}

\subsubsection{Model Setup and Computational Cost}
\label{sec:model}
In this work, we adopt shallow feedforward neural networks.
More specifically, for problems with periodic boundary conditions, 
we enforce the boundary condition by incorporating a periodic Fourier feature layer \cite{DongMethod2021JCP} in the first 
layer of the network. For network parameter initialization, we employ a uniform distribution and set $R_\text{init}=1$ in the general setting. 
In addition, for the 2D Cahn--Hilliard equation and the 1D Allen--Cahn equation, 
we employ a multi-scale neural network architecture (MsDNN) \cite{LiMultiscale2020C} to better capture features across 
different spatial scales.

The computational cost for the examples with different settings are as follows: 
\begin{table}[h]
  \caption{Computational cost for the examples, the numbers in parentheses reported 
  in the runtime indicate the time cost on the initial construction of the Jacobian matrix.}
  \centering
  \begin{tabular}{lcccccr}
    \toprule
      & MsDNN & Activation & Width & Data Size & Iterations & Runtime (s) \\
    \midrule
    Example~\ref{ex:1}& \ding{51} & $\tanh$ & $1000$ & 4000 & 200 & 101 (7) \\
    Example~\ref{ex:2}& \ding{55} & $\tanh$ & $1000$ & 4000 & 400 & 155 (17) \\
    Example~\ref{ex:3}& \ding{55} & $\tanh$ & $1000$ & 4000 & 500 & 174 (16) \\
    Example~\ref{ex:4}& \ding{51} & $\sin$ & $1000$ & 4096 & 1600 & 948 (66) \\
    % Example~\ref{ex:4}& + & $2000$ & 10000 & 1600 & 4016 (355) \\
    Example~\ref{ex:5}& \ding{55} & $\tanh$ & $1000$ & 9128 & 180 & 62 (15) \\
    Example~\ref{ex:6}& \ding{51} & $\tanh$ & $1000$ & 6412 & 20 & 86 (57) \\
  \end{tabular}
  \label{tab:runtime}
  \vskip -0.1in
\end{table}

\subsubsection{Visualization Results for Examples}
\label{sec:visualresults}
We show the visualization results for all the six examples: 
\begin{figure}[h]
  % \vskip 0.2in
  \centering
    \includegraphics[width=0.9\textwidth]{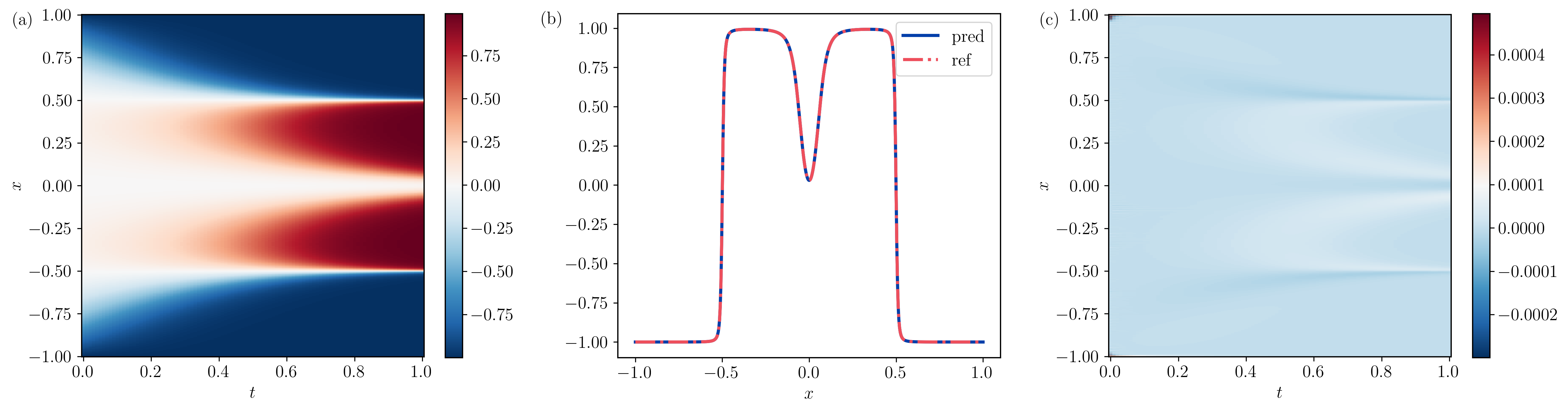}
    \caption{
    Visualization of Example~\ref{ex:1}: (a) global predicted solution; (b) predicted and reference solutions at the final time;
     (c) global error. 
    }
  \label{fig:ex1}
\end{figure}

\begin{figure}[h]
  % \vskip 0.2in
  \centering
    \includegraphics[width=0.9\textwidth]{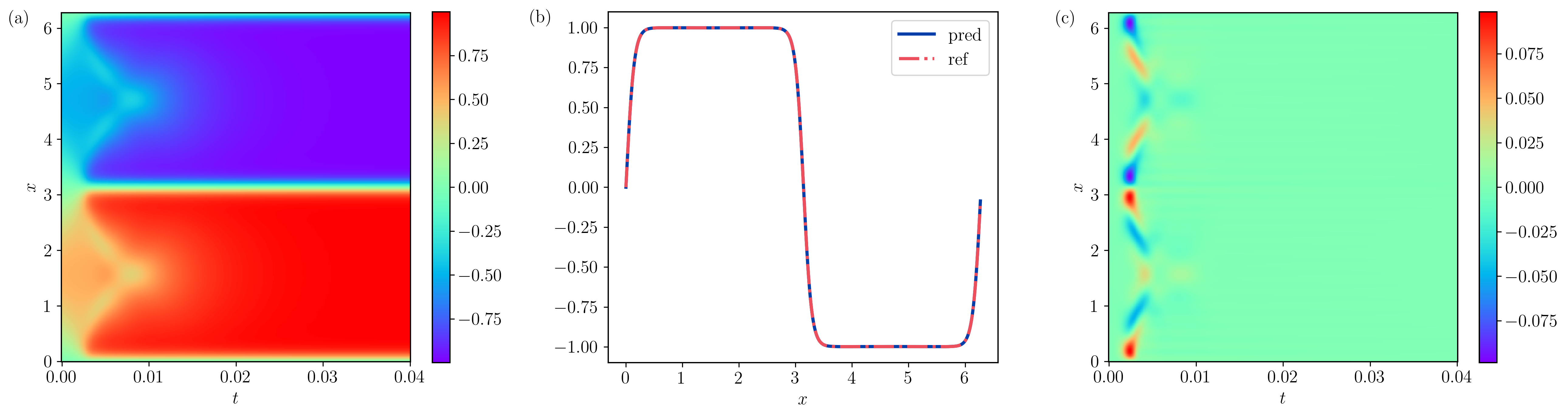}
    \caption{
    Visualization of Example~\ref{ex:2}: (a) global predicted solution; (b) predicted and reference solutions at the final time;
     (c) global error. 
    }
  \label{fig:ex2}
\end{figure}

\begin{figure}[h]
  % \vskip 0.2in
  \centering
    \includegraphics[width=0.9\textwidth]{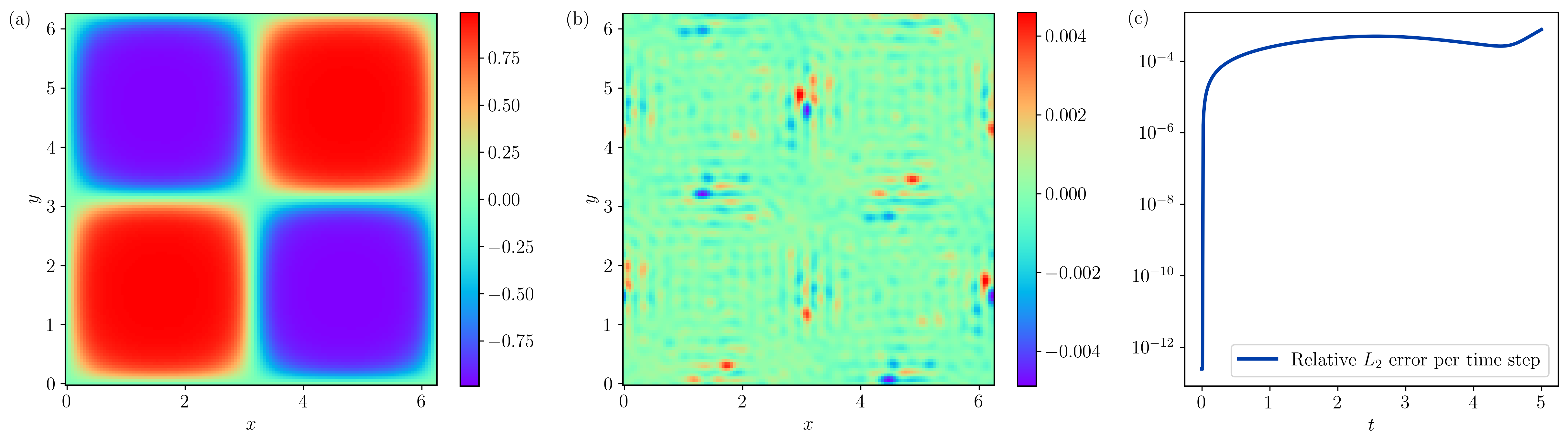}
    \caption{
    Visualization of Example~\ref{ex:3}: (a) predicted and reference solutions at the final time; (b) error at the final time; 
    (c) relative $L^2$ errors at each time step.
    }
  \label{fig:ex3}
\end{figure}

\begin{figure}[h]
  % \vskip 0.2in
  \centering
    \includegraphics[width=0.9\textwidth]{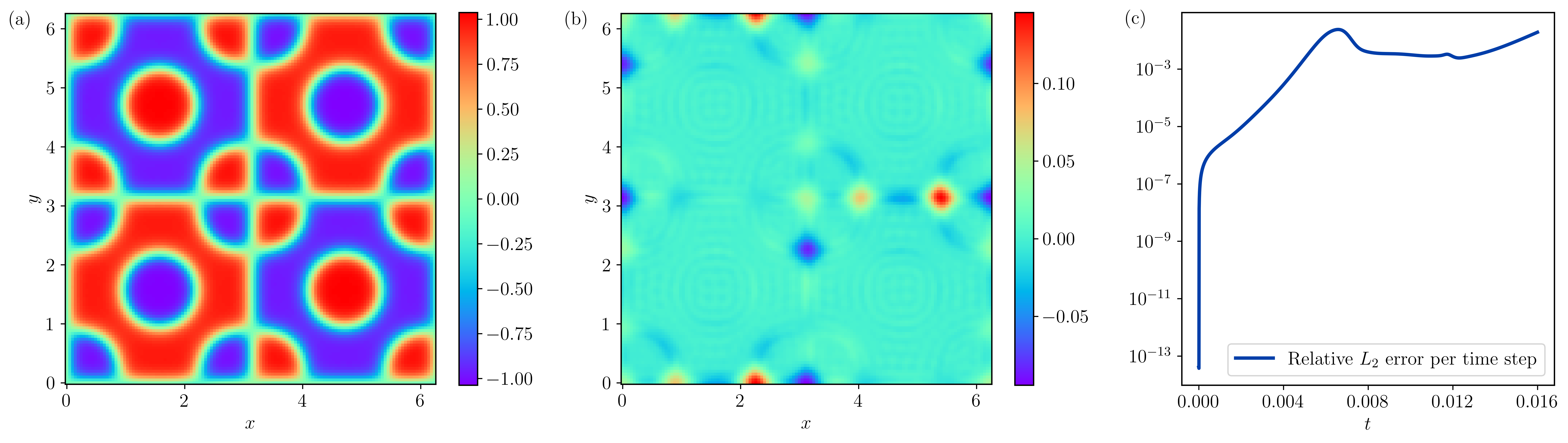}
    \caption{
    Visualization of Example~\ref{ex:4}: (a) predicted and reference solutions at the final time; (b) error at the final time; 
    (c) relative $L^2$ errors at each time step.
    }
  \label{fig:ex4}
\end{figure}

\begin{figure}[t]
  % \vskip 0.2in
  \centering
    \includegraphics[width=0.9\textwidth]{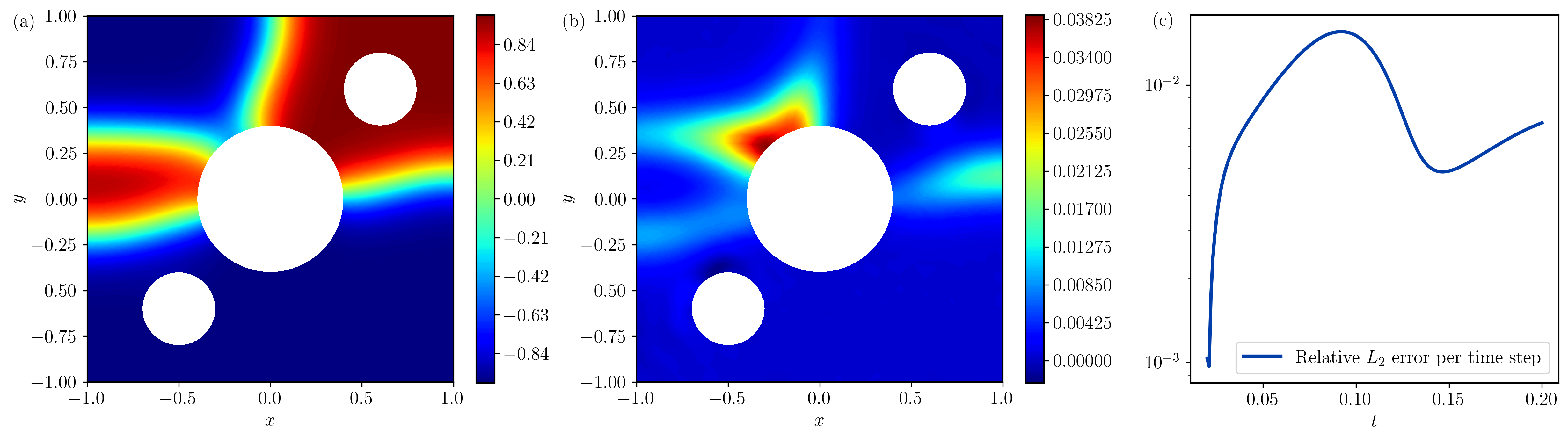}
    \caption{
    Visualization of Example~\ref{ex:5}: (a) predicted and reference solutions at the final time; (b) error at the final time; 
    (c) relative $L^2$ errors at each time step.
    }
  \label{fig:ex5}
\end{figure}

\begin{figure}[t]
  % \vskip 0.2in
  \centering
    \includegraphics[width=0.9\textwidth]{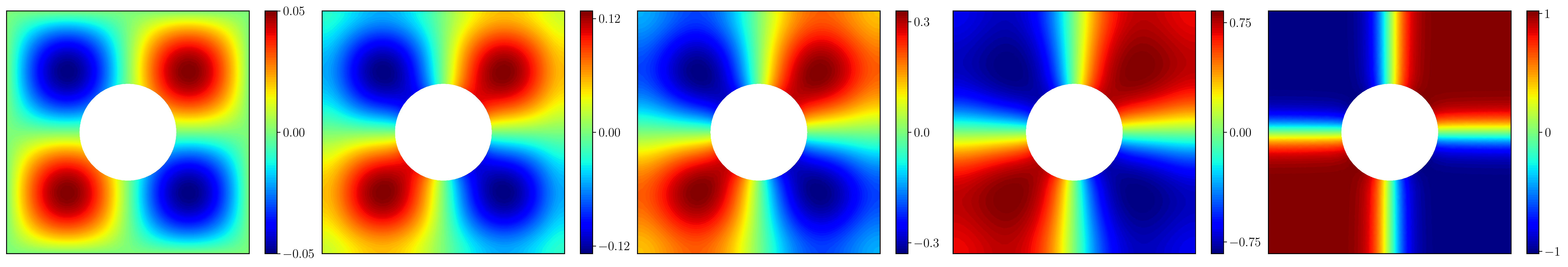}
    \caption{
    Evolution of Example~\ref{ex:6}: predictions at $t=0$, $t=0.03$, $t=0.06$, $t=0.1$, $t=0.15$
    }
  \label{fig:ex6}
\end{figure}

\begin{figure}[t]
  % \vskip 0.2in
  \centering
    \includegraphics[width=0.9\textwidth]{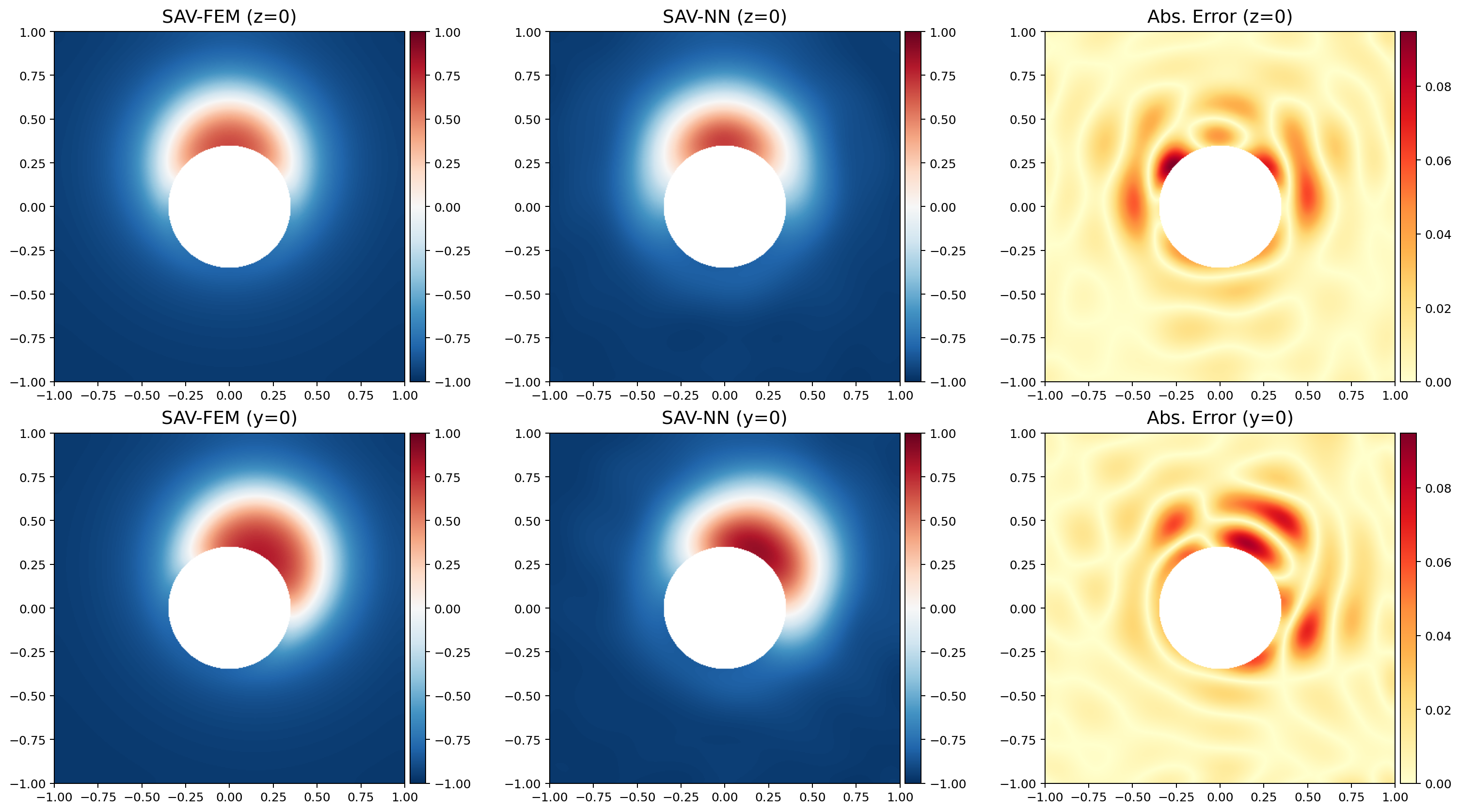}
    \caption{
    Performance of Example~\ref{ex:7}: predictions on the slices $y=0$ and $z=0$ at $t=0.02$. The left panel shows the result computed using the SAV-based finite element method, the middle panel shows the result obtained by our method, and the right panel shows the absolute error. The overall relative error is approximately 0.01.
    }
  \label{fig:ex7}
\end{figure}

\end{document}